\documentclass[11pt]{article}
\usepackage[top=1in, bottom=1in, left=1in, right=1in]{geometry}

\usepackage[linktocpage,colorlinks,linkcolor=blue,anchorcolor=blue,citecolor=blue,urlcolor=blue,pagebackref]{hyperref}
\usepackage{euscript,mdframed}
\usepackage{microtype,todonotes,relsize}
\usepackage{amsmath,amsthm}
\usepackage{algorithmic}

\usepackage{accents}
\usepackage{comment,url,graphicx,relsize}
\usepackage{amssymb,amsfonts,amsmath,amsthm,amscd,dsfont,mathrsfs,mathtools,nicefrac, bm}
\usepackage{float,psfrag,epsfig,color,xcolor,url}
\usepackage{epstopdf,bbm,mathtools,enumitem}
\usepackage{subfigure}
\usepackage[ruled,vlined]{algorithm2e}
\usepackage{tablefootnote}
\graphicspath{{Plots/}}
\usepackage{diagbox}
\usepackage{tikz}
\usetikzlibrary{calc,patterns,angles,quotes}
\usepackage{makecell}
\usepackage{multirow}

\newcommand{\grad}{\mathsf{grad}} 

\renewcommand{\d}{\mathrm{d}}
\newcommand{\M}{\mathcal{M}}
\newcommand{\N}{\mathcal{N}}
\newcommand{\D}{\mathrm{D}}

\newcommand{\E}{\mathbb{E}}
\newcommand{\RR}{\mathbb{R}}

\providecommand{\logdet}{\mathop\mathrm{log det}}
\providecommand{\argmin}{\mathop\mathrm{arg min}}

\providecommand{\tr}{\mathop\mathrm{tr}}

\providecommand{\sym}{\mathop\mathrm{sym}}

\newcommand{\Exp}{\mathsf{Exp}}
\newcommand{\proj}{\mathsf{proj}}
\newcommand{\dist}{\operatorname{dist}}

\newcommand{\T}{\operatorname{T}}

\ifdefined\nonewproofenvironments\else
\ifdefined\ispres\else
\renewenvironment{proof}{\noindent\textbf{Proof.}\hspace*{.3em}}{\qed\\}
\newenvironment{proof-sketch}{\noindent\textbf{Proof Sketch}
  \hspace*{0.em}}{\qed\bigskip\\}
\newenvironment{proof-idea}{\noindent\textbf{Proof Idea}
  \hspace*{0.em}}{\qed\bigskip\\}
\newenvironment{proof-of-lemma}[1][{}]{\noindent\textbf{Proof of Lemma {#1}.}
  \hspace*{0.em}}{\qed\\}
\newenvironment{proof-of-corollary}[1][{}]{\noindent\textbf{Proof of Corollary {#1}.}
  \hspace*{0.em}}{\qed\\}
\newenvironment{proof-of-theorem}[1][{}]{\noindent\textbf{Proof of Theorem {#1}.}
  \hspace*{0.em}}{\qed\\}
\newenvironment{proof-attempt}{\noindent\textbf{Proof Attempt}
  \hspace*{0.em}}{\qed\bigskip\\}

\fi
\newtheorem{theorem}{Theorem}[section]
\newtheorem{lemma}{Lemma}[section]

\newtheorem{proposition}{Proposition}[section]
\newtheorem{assumption}{Assumption}[section]
\newtheorem{remark}{Remark}[section]

\newtheorem{definition}{Definition}[section]

\usetikzlibrary{calc}

\allowdisplaybreaks
\usepackage[authoryear,round]{natbib}
\renewcommand*{\backref}[1]{\ifx#1\relax \else Page #1 \fi}
\renewcommand*{\backrefalt}[4]{%
  \ifcase #1 \footnotesize{(Not cited.)}%
  \or        \footnotesize{(Cited on page~#2.)}%
  \else      \footnotesize{(Cited on pages~#2.)}%
  \fi
}

\newcommand*{\colorboxed}{}
\def\colorboxed#1#{%
  \colorboxedAux{#1}%
}
\newcommand*{\colorboxedAux}[3]{%
  \begingroup
    \colorlet{cb@saved}{.}%
    \color#1{#2}%
    \boxed{%
      \color{cb@saved}%
      #3%
    }%
  \endgroup
}
\numberwithin{equation}{section}
\usepackage{xcolor} 
\usepackage{marginnote}

\newcommand{\todol}[2][]{{%
 \let\marginpar\marginnote
 \reversemarginpar
 \renewcommand{\baselinestretch}{0.8}%
 \todo[color=yellow]{#2}}}

\title{Riemannian Bilevel Optimization}
\author{
{Jiaxiang Li} \thanks{Department of Electrical and Computer Engineering, University of Minnesota, Twin Cities.  \texttt{li003755@umn.edu}}
\and
Shiqian Ma \thanks{Department of Computational Applied Math and Operations Research, Rice University. Research supported in part by NSF grants DMS-2243650, CCF-2308597, CCF-2311275 and ECCS-2326591, and a startup fund from Rice University. \texttt{sqma@rice.edu}}
\and
}
\date{Feb 2, 2024}

\begin{document}

\maketitle

\begin{abstract}
    In this work, we consider the bilevel optimization problem on Riemannian manifolds. We inspect the calculation of the hypergradient of such problems on general manifolds and thus enable the utilization of gradient-based algorithms to solve such problems. The calculation of the hypergradient requires utilizing the notion of Riemannian cross-derivative and we inspect the properties and the numerical calculations of Riemannian cross-derivatives. Algorithms in both deterministic and stochastic settings, named respectively RieBO and RieSBO, are proposed that include the existing Euclidean bilevel optimization algorithms as special cases. Numerical experiments on robust optimization on Riemannian manifolds are presented to show the applicability and efficiency of the proposed methods.
\end{abstract}

\section{Introduction}
Bilevel optimization has drawn attentions from various fields in optimization and machine learning communities, due to its wide range of applications including meta learning~\citep{rajeswaran2019meta,ji2020convergence}, hyperparameter optimization~\citep{okuno2021lp,yu2020hyper}, reinforcement learning~\citep{konda1999actor,hong2020two} and signal processing~\citep{kunapuli2008classification,flamary2014learning}. In this work, we focus on the manifold-constrained bilevel optimization problem, which can be formulated as:
\begin{equation}\label{deterministic_problem}
\begin{aligned}
    & \min_{x\in\M} \Phi(x):= f(x, y^*(x)) \\
    & \text{ s.t. }y^*(x)=\argmin_{y\in\N} g(x,y),
\end{aligned} 
\end{equation}
where $\M$ and $\N$ are $m$ and $n$-dimensional complete Riemannian manifolds, respectively. We also consider the stochastic bilevel optimization which is in the following form:
\begin{equation}\label{stochastic_problem}
\begin{aligned}
&\min _{x \in \M} \Phi(x)=f\left(x, y^{*}(x)\right):=\mathbb{E}_{\xi}\left[F\left(x, y^{*}(x) ; \xi\right)\right] \\
&\text { s.t. } y^{*}(x)=\underset{y \in \N}{\argmin}\ g(x, y):=\mathbb{E}_{\zeta}[G(x, y ; \zeta)],
\end{aligned}
\end{equation}
where $\xi$ and $\zeta$ are random variables that usually represent the randomness from the data. Such a framework allows us to utilize the stochastic gradient methods to get a desired convergence result with only a noisy estimate of the gradients for $f$ and $g$. Here we also assume that $g$ is a (geodesically) strongly convex function with respect to $y$ in both \eqref{deterministic_problem} and \eqref{stochastic_problem} so that the solution to the lower level problem $y^*(x)$ is well-defined. 


Notice that the original (Euclidean) bilevel optimization is a special case of \eqref{deterministic_problem} by taking the manifolds as the Euclidean spaces with the same dimensions:
\begin{equation}\label{bilevel_problem}
\begin{aligned}
    & \min_{x\in\RR^m} \Phi(x):= f(x, y^*(x)) \\
    & \text{ s.t. }y^*(x)=\argmin_{y\in\RR^n} g(x,y),
\end{aligned}
\end{equation}
where $f$ and $g$ are assumed to be continuously differentiable. It is worth noticing that the objective function $\Phi(x)$ is still nonconvex even if we impose convexity assumptions on $f$, which makes such a problem hard to tackle, let alone the more complicated manifold-constraint problems, namely \eqref{deterministic_problem} and \eqref{stochastic_problem}.

There has been an extensive study on the Euclidean bilevel optimization~\citep{ji2021bilevel,hong2020two,chen2021tighter,ghadimi2018approximation}. On the algorithmic sense, the bilevel optimization seeks to obtain a first-order $\epsilon$-stationary point (Definition \ref{def_deterministic_eps_stationary} and \ref{def_stochastic_eps_stationary} for deterministic and stochastic cases, respectively) with the access to the gradient oracle of $f$ and $g$, as well as the Jacobian- and Hessian-vector product, i.e. $\nabla_x\nabla_y g(x,y)v$ and $\nabla_{y}^2 g(x,y)v$, respectively. To find an $\epsilon$-stationary point, we denote the number of calls to the gradient oracle of $f$ and $g$ as $\operatorname{Gc}(f,\epsilon)$ and $\operatorname{Gc}(g,\epsilon)$, correspondingly; similarly we have the notation $\operatorname{JV}$ and $\operatorname{HV}$ for the number of oracle calls for the Jacobian- and Hessian-vector product. In the Euclidean setting, We have Tables \ref{table0} and \ref{table1} as the summary for the oracle calls to achieve an $\epsilon$-stationary point for deterministic and stochastic cases, correspondingly (and we denote $\kappa$ as the condition number of the lower level strongly convex problem).

\begin{table}[ht]
\begin{center}
\begin{small}
\begin{tabular}{c || c | c | c} 
 \hline
 Algorithm & BA & AID-BiO & ITD-BiO* \\
 \hline\hline
 $y$-update & GD & GD & GD \\
 \hline
 $\operatorname{Gc}(f,\epsilon)$ & $\mathcal{O}(\kappa^{4}\epsilon^{-1})$ & $\mathcal{O}(\kappa^{3}\epsilon^{-1})$ & $\mathcal{O}(\kappa^{3}\epsilon^{-1})$ \\
 \hline
 $\operatorname{Gc}(g,\epsilon)$ & $\mathcal{O}(\kappa^{5}\epsilon^{-5/4})$ & $\mathcal{O}(\kappa^{4}\epsilon^{-1})$ & $\mathcal{O}(\kappa^{4}\epsilon^{-1})$ \\
 \hline
 $\operatorname{JV}(g,\epsilon)$ & $\mathcal{O}(\kappa^{4}\epsilon^{-1})$ & $\mathcal{O}(\kappa^{3}\epsilon^{-1})$ & $\mathcal{O}(\kappa^{4}\epsilon^{-1})$ \\
 \hline
 $\operatorname{HV}(g,\epsilon)$ & $\mathcal{O}(\kappa^{4.5}\epsilon^{-1})$ & $\mathcal{O}(\kappa^{3.5}\epsilon^{-1})$ & $\mathcal{O}(\kappa^{4}\epsilon^{-1})$ \\
 \hline
\end{tabular}

\footnotesize{* Require explicit assumption on the sequence.}
\end{small}
\end{center}
\caption{Summary of the convergence results for different algorithms for \textbf{deterministic} Euclidean bilevel optimization, including BA~\citep{ghadimi2018approximation}, AID-BiO~\citep{ji2021bilevel} and ITD-BiO~\citep{ji2021bilevel}.}
\label{table0}
\end{table}

\begin{table}[ht]
\begin{center}
\begin{small}
\begin{tabular}{c || c | c | c | c | c} 
 \hline
 Algorithm & BSA & Stoc-BiO & TTSA* & ALSET & STABLE* \\
 \hline\hline
 batch size & $\mathcal{O}(1)$ & $\mathcal{O}(\epsilon^{-1})$ & $\mathcal{O}(1)$ & $\mathcal{O}(1)$ & $\mathcal{O}(1)$ \\ 
 \hline
 $y$-update & $\mathcal{O}(\epsilon^{-1})$ steps SGD & SGD & SGD & SGD & correction\\
 \hline
 $\operatorname{Gc}(F,\epsilon)$ & $\mathcal{O}(\kappa^{6}\epsilon^{-2})$ & $\mathcal{O}(\kappa^{5}\epsilon^{-2})$ & $\mathcal{O}(\operatorname{poly}(\kappa)\epsilon^{-2.5})$ & $\mathcal{O}(\kappa^{5}\epsilon^{-2})$ & $\mathcal{O}(\operatorname{poly}(\kappa)\epsilon^{-2})$ \\
 $\operatorname{Gc}(G,\epsilon)$ & $\mathcal{O}(\kappa^{9}\epsilon^{-2})$ & $\mathcal{O}(\kappa^{9}\epsilon^{-2})$ & $\mathcal{O}(\operatorname{poly}(\kappa)\epsilon^{-2.5})$ & $\mathcal{O}(\kappa^{9}\epsilon^{-2})$ & $\mathcal{O}(\operatorname{poly}(\kappa)\epsilon^{-2})$ \\
 $\operatorname{JV}(G,\epsilon)$ & $\mathcal{O}(\kappa^{6}\epsilon^{-2})$ & $\mathcal{O}(\kappa^{5}\epsilon^{-2})$ & $\mathcal{O}(\operatorname{poly}(\kappa)\epsilon^{-2.5})$ & $\mathcal{O}(\kappa^{5}\epsilon^{-2})$ & $\mathcal{O}(\operatorname{poly}(\kappa)\epsilon^{-2})$ \\
 $\operatorname{HV}(G,\epsilon)$ & $\mathcal{O}(\kappa^{6}\epsilon^{-2})$ & $\mathcal{O}(\kappa^{6}\epsilon^{-2})$ & $\mathcal{O}(\operatorname{poly}(\kappa)\epsilon^{-2.5})$ & $\mathcal{O}(\kappa^{6}\epsilon^{-2})$ & $\mathcal{O}(\operatorname{poly}(\kappa)\epsilon^{-2})$ \\
 \hline
\end{tabular}

\footnotesize{* For algorithms that did not specify the dependence on condition number $\kappa$, we use the notation $\operatorname{poly}(\kappa)$ to summarize the $\kappa$ dependence.}
\end{small}
\end{center}
\caption{Summary of the convergence results for different algorithms for \textbf{stochastic} Euclidean bilevel optimization, including BSA~\citep{ghadimi2018approximation}, Stoc-BiO~\citep{ji2021bilevel}, TTSA~\citep{hong2020two}, ALSET~\citep{chen2021tighter} and STABLE~\citep{chen2021single}. For the batch size we only include the $\epsilon$ dependency.}
\label{table1}
\end{table}

\subsection{Main results}
In this work, we first analyze the method of calculating and estimating the hypergradient for bilevel problems on Riemannian manifolds. Our propositions include the Euclidean bilevel problems as special cases and involve the calculation of Riemannian cross derivatives, which are of independent interests to the Riemannian optimization field.

Our contribution also lies in proposing two algorithms (RieBO and RieSBO) for both the problems \eqref{deterministic_problem} and \eqref{stochastic_problem} correspondingly. For the deterministic problem (\ref{deterministic_problem}), our analysis shows that with a multi-step inner loop and a single-step outer loop, one could yield the similar gradient complexities $\operatorname{Gc}(f,\epsilon)$, $\operatorname{Gc}(g,\epsilon)$, Jacobian- and Hessian-vector product complexities $\operatorname{JV}(g,\epsilon)$ and $\operatorname{HV}(g,\epsilon)$ same as the Euclidean counterparts in~\cite{ji2021bilevel}, as presented in Table \ref{table3}, al well as for the stochastic problem (\ref{stochastic_problem})~\citep{chen2021tighter}. It is worth noticing that for the stochastic problem, we adopt the framework of~\cite{chen2021tighter} onto Riemannian manifolds so that the batch-size of the hypergradient estimate can be $\mathcal{O}(1)$, significantly smaller than $\mathcal{O}(\epsilon^{-1})$ as in~\cite{ji2021bilevel}.

\begin{table}[ht]
\begin{center}
\begin{small}
\begin{tabular}{c || c | c} 
 \hline
 Algorithm & RieBO (Algorithm \ref{algo_bilevel_AID_ITD}) & RieSBO (Algorithm \ref{algo_stoc_bilevel_AID_ITD}) \\
 \hline\hline
 batch size & No batch & $\mathcal{O}(1)$ \\ 
 \hline
 $y$-update & GD & SGD \\
 \hline
 $\operatorname{Gc}(F,\epsilon)$ & $\mathcal{O}(\kappa^{3}\epsilon^{-1})$ & $\mathcal{O}(\kappa^{5}\epsilon^{-2})$ \\
 $\operatorname{Gc}(G,\epsilon)$ & $\mathcal{O}(\kappa^{4}\epsilon^{-1})$ & $\mathcal{O}(\kappa^{9}\epsilon^{-2})$ \\
 $\operatorname{JV}(G,\epsilon)$ & $\mathcal{O}(\kappa^{3}\epsilon^{-1})$ & $\mathcal{O}(\kappa^{5}\epsilon^{-2})$ \\
 $\operatorname{HV}(G,\epsilon)$ & $\mathcal{O}(\kappa^{3.5}\epsilon^{-1})$ & $\mathcal{O}(\kappa^{6}\epsilon^{-2})$ \\
 \hline
\end{tabular}
\end{small}
\end{center}
\caption{Summary of the convergence results for the proposed algorithms in this paper, where all the oracles are with respect to Riemannian gradients and Riemannian second-order derivatives. RieBO (Algorithm \ref{algo_bilevel_AID_ITD}) solves \eqref{deterministic_problem}, and RieSBO (Algorithm \ref{algo_stoc_bilevel_AID_ITD}) solves \eqref{stochastic_problem}.}
\label{table3}
\end{table}

Finally, we implement the proposed method in the manifold-constrained bilevel optimization problems, namely the distributionally robust optimization on Riemannian manifolds with two specific examples: robust maximum likelihood estimation and robust Karcher mean problem on the manifold of positive definite matrices. These numerical results demonstrate the efficiency and potential applicability of the proposed methods.

\subsection{Related works}
\textbf{Bilevel optimization}. Bilevel optimization problem, also known as nested optimization problem, whose origin dates back to the 50s and 70s~\citep{stackelberg1952theory,bracken1973mathematical}. Since then, extensive studies have been conducted for solving the bilevel optimization problem~\citep{shi2005extended,moore2010bilevel}. Recently, gradient-based algorithms for solving bilevel optimization problems draw attention because of their applications in machine learning, see, e.g.~\cite{domke2012generic,pedregosa2016hyperparameter,gould2016differentiating,maclaurin2015gradient,franceschi2018bilevel,liao2018reviving,shaban2019truncated,liu2020generic,li2020improved,grazzi2020iteration,lorraine2020optimizing,ji2021lower,ghadimi2018approximation,hong2020two,ji2021bilevel,chen2021tighter,chen2021single}, among which there has been discussions about the rate of convergence for specific algorithms~\citep{ghadimi2018approximation,hong2020two,ji2021bilevel,chen2021single,chen2021tighter}. These well-established convergence rate results are summarized in Tables \ref{table0} and \ref{table1}. 
{Recently, a line of work~\citep{khanduri2021near,yang2023achieving} which utilizes the momentum-based stochastic algorithms can achieve a better oracle complexity of $\mathcal{O}(\epsilon^{-1.5})$ for the Euclidean version of the stochastic problem \eqref{stochastic_problem}. We did not include this line of work since the Riemannian counterparts of these works would rely on the utilization of parallel/vector transport in the algorithm updates. We deliberately avoid these complicated operations in algorithm design and postpone them for future works.}

It is worth mentioning that minimax saddle point problems $\min_x\max_y f(x,y)$ are special cases of bilevel optimization problems by taking $g=-f$. Minimax problems are of great interests to the machine learning community \citep{daskalakis2018limit,mokhtari2020unified,yoon2021accelerated,lin2020gradient}. The analysis of this paper relates to the nonconvex-strongly-concave minimax problem on Riemannian manifolds as in \cite{huang2020gradient}, which showed that the Riemannian gradient descent ascent (RGDA) achieves oracle calls with orders $\mathcal{O}(\kappa^2\epsilon^{-1})$ for the deterministic case and $\mathcal{O}(\kappa^3\epsilon^{-2})$ for the stochastic case. These results match our convergence results in terms of the order of $\epsilon$, but has better $\kappa$ dependence. This makes sense because our proposed method is a multi-$y$ step GDmax algorithm (see~\citep{nouiehed2019solving,jin2020local}) when applied to the minimax problem and naturally has a larger $\kappa$ dependence. Recently, the authors of \cite{cai2023curvature} considered the minimax game on Riemannian manifolds under the assumption of geodesic-strongly-monotone (a generalization of strongly-convex-strongly-concave minimax game) and provided a stochastic Riemannian gradient descent-ascent approach which enjoys linear rate of convergence -- similar to its Euclidean counterpart. We point out that our work considers nonconvex upper level problems, which is different from the setting in \cite{cai2023curvature}.

Other bilevel-related ongoing research topics include decentralized bilevel optimization~\cite{chen2022decentralized,chen2023decentralized,dong2023single}, federate bilevel optimization~\cite{tarzanagh2022fednest}, bilevel without lower strongly convexity~\cite{chen2023bilevel}, to name a few.

\textbf{Optimization on Riemannian manifolds}. Optimization on Riemannian manifolds draws lots of attention recently due to its applications in various fields, including low-rank matrix completion~\citep{boumal2011rtrmc,vandereycken2013low}, phase retrieval~\citep{bendory2017non,sun2018geometric}, dictionary learning~\citep{cherian2016riemannian,sun2016complete}, dimensionality reduction~\citep{harandi2017dimensionality,tripuraneni2018averaging,mishra2019riemannian} and manifold regression~\citep{lin2017extrinsic,lin2020robust}. The manifold optimization usually transforms an manifold constrained problem into an unconstrained problem by viewing the manifold as the ambient space and using proper retraction to deal with the loss of linearity, thus achieves better convergence results. For smooth Riemannian optimization, it can be shown that Riemannian gradient descent method require $\mathcal{O}(1/{\epsilon})$ iterations to converge to an $\epsilon$-stationary point (i.e. bounding the norm square of the gradient by $\epsilon$)~\citep{boumal2018global}. Stochastic algorithms were also studied for smooth Riemannian optimization~\citep{bonnabel2013stochastic,zhou2019faster,weber2019nonconvex,zhang2016fast,kasai2018riemannian}.

The combination of bilevel optimization with Riemannian optimization is largely blank. \cite{bonnel2015semivectorial} considered a semi-vectorial bilevel optimization model over Riemannian manifolds, which deals with the situation where the lower level problem does not have unique solutions and necessary optimality conditions are provided for their surrogate model. To the best of our knowledge, there lacks convergence analysis for Riemannian bilevel optimization in the literature.

\section{Preliminaries on Riemannian Optimization}

In this part, we briefly review the basic tools we use for optimization on Riemannian manifolds~\citep{lee2006riemannian,Tu2011manifolds,boumal2022intromanifolds}. Suppose $\M$ is an $m$-dimensional differentiable manifold. The tangent space $\T_x\M$ at $x\in\M$ is a linear subspace that consists of the derivatives of all differentiable curves on $\M$ passing through $x$: $\T_x\M:=\{\gamma^{\prime}(0): \gamma(0)=x, \gamma([-\delta, \delta]) \subset \mathcal{M}\text { for some } \delta>0, \gamma \text { is differentiable}\}$. Notice that for every vector $\gamma^{\prime}(0)\in \T_x\M$, it can be defined in a coordinate-free sense via the operation over smooth functions: $\forall f\in C^{\infty}(\M)$, $\gamma^{\prime}(0)(f):=\frac{d f\circ \gamma(t)}{dt}\mid_{t=0}$. The notion of Riemannian manifold is defined as follows.

\begin{definition}[Riemannian manifold]
    Manifold $\M$ is a Riemannian manifold if it is equipped with an \textbf{inner product} on the tangent space, $\langle \cdot, \cdot \rangle_x : \T_x\M \times \T_x\M \rightarrow \RR$, that varies smoothly on $\M$. The $(0,2)$-tensor field $g$ is usually referred to as Riemannian metric.
\end{definition}

We also review the notion of the differential between manifolds here.
\begin{definition}[Differential and Riemannian gradients]
    Let $F:\M\rightarrow \mathcal{N}$ be a $C^{\infty}$ map between two differential manifolds. At each point $x\in\M$, the differential of $F$ is a mapping (also known as the push-forward):
    $$
        \D F:\T_x\M\rightarrow \T_{F(x)}\mathcal{N},
    $$
    such that $\forall\xi\in \T_x\M$, $\D F(\xi)\in \T_x\mathcal{N}$ is given by
    $$
        (\D F(\xi))(f):=\xi(f\circ F)\in\RR,\ \forall f\in C_{F(x)}^{\infty}(\M).
    $$
    
    If $\mathcal{N}=\RR$, i.e. $f\in C^\infty(\M)$, the differential of $f$ is usually denoted as $\d f$. For a Riemannian manifold with Riemannian metric $\langle\cdot, \cdot\rangle$, the Riemannian gradient for $f\in C^\infty(\M)$ is the unique tangent vector $\grad f(x)\in \T_x\M$ satisfying
    $$
        \d f(\xi) = \langle\grad f, \xi\rangle_x,\ \forall \xi\in \T_x\M.
    $$
\end{definition}

For the convergence analysis, we also need the notion of exponential mapping and parallel transport. To this end, we need to first recall the definition of a geodesic.
\begin{definition}[Geodesic and exponential mapping]
    Given $x\in\M$ and $\xi\in \T_x\M$, the geodesic is the curve $\gamma:I\rightarrow\M$, $0\in I\subset\RR$ is an open set, so that $\gamma(0)=x$, $\Dot{\gamma}(0)=\xi$ and $\nabla_{\Dot{\gamma}}\Dot{\gamma}=0$ where $\nabla:\T_x\M\times \T_x\M\rightarrow \T_x\M$ is the Levi-Civita connection defined by metric $g$. In local coordinate sense, $\gamma$ is the unique solution of the following second-order differential equations:
    $$
        \frac{d^2\gamma^k}{d t^2} + \Gamma_{i,j}^{k}\frac{d\gamma^i}{d t}\frac{d\gamma^j}{d t}=0,
    $$
    under Einstein summation convention, where $\Gamma_{i,j}^{k}$ are Christoffel symbols, again defined by metric tensor. The exponential mapping $\Exp_x$ is defined as a mapping from $\T_x\M$ to $\M$ s.t. $\Exp_x(\xi):= \gamma(1)$ with $\gamma$ being the geodesic with $\gamma(0)=x$, $\Dot{\gamma}(0)=\xi$. A natural corollary is $\Exp_x(t\xi):= \gamma(t)$ for $t\in[0, 1]$. Another useful fact is $\dist(x,\Exp_x(\xi))=\|\xi\|_x$ since $\gamma'(0)=\xi$ which preserves the speed. Here $\dist$ is the geodesic distance which connects the two points by the minimum geodesic.
\end{definition}
Throughout this paper, we always assume that $\M$ is complete, so that $\Exp_x$ is always defined for every $\xi\in \T_x\M$. For any $x,y\in\M$, the inverse of the exponential mapping $\Exp_{x}^{-1}(y)\in \T_x\M$ is called the logarithm mapping, and we have $\dist(x,y)=\|\Exp_{x}^{-1}(y)\|_x$, which derives directly from $\dist(x,\Exp_x(\xi))=\|\xi\|_x$.

With the notion of geodesic, we have the following definition of geodesic convexity and strong convexity, which are the generalizations of their Euclidean counterparts.
\begin{definition}[Geodesic (strong) convexity]
    A geodesic convex set $\Omega\subset\M$ is a set such that for any two points in the set, there exists a geodesic connecting them that lies entirely in $\Omega$. A function $h:\Omega\rightarrow \RR$ is called geodesically convex if for any $p, q\in\Omega$, we have $h(\gamma(t))\leq (1-t)h(p)+th(q)$, where $\gamma$ is a geodesic in $\Omega$ with $\gamma(0)=p$ and $\gamma(1)=q$. It is called $\mu$-geodesically strongly convex if we have $h(\gamma(t))\leq (1-t)h(p)+th(q)-\frac{\mu t(1-t)}{2}\dist(p, q)^2$.

    If $h$ is a continuously differentiable function, then it is geodesically convex if and only if (see \citet[Chapter 11]{boumal2022intromanifolds}) $h(y)\geq h(x) + \langle\grad h(x), \Exp^{-1}_{x}(y)\rangle_{x}$, and is geodesically strongly convex if and only if $h(y)\geq h(x) + \langle\grad h(x), \Exp^{-1}_{x}(y)\rangle_{x} + \frac{\mu}{2}\dist(x, y)^2$.

    If $h$ is a twice continuously differentiable function, then it is geodesically convex if and only if (see \citet[Chapter 11]{boumal2022intromanifolds}) $\frac{d^2 h(\gamma(t))}{d t^2}\geq 0$, and is geodesically strongly convex if and only if $\frac{d^2 h(\gamma(t))}{d t^2}\geq \mu$.
\end{definition}

We also present the definition of parallel transport, which is used in the assumption and the convergence analysis, but not explicitly used in the algorithm updates.
\begin{definition}[Parallel transport]
    Given a Riemannian manifold $(\M, g)$ and two points $x,y\in\M$, the parallel transport $P_{x\rightarrow y}:\T_x\M\rightarrow \T_y\M$ is a linear operator which keeps the inner product: $\forall \xi,\zeta\in \T_x\M$, we have $\langle P_{x\rightarrow y}\xi, P_{x\rightarrow y}\zeta\rangle_y = \langle\xi, \zeta\rangle_x$.
\end{definition}
Notice that the existence of parallel transport depends on the curve connecting $x$ and $y$, which is not a problem for complete Riemannian manifold since we always take the unique geodesic that connects $x$ and $y$. Parallel transport is useful in our convergence proofs since the Lipschitz condition for the Riemannian gradient requires moving the gradients in different tangent spaces ``parallel'' to the same tangent space. 

We also have the following definition of Lipschitz smoothness on the manifolds.
\begin{definition}[Geodesic Lipschitz smoothness]
    A function $h:\Omega\rightarrow \RR$ is called geodesic-Lipschitz smooth if we have:
    \begin{equation}
        \|\grad h(y) - P_{x\rightarrow y}\grad h(x)\| \leq \ell_{h}\dist(x, y).
    \end{equation}
Moreover, we have (see \cite{zhang2016fast})
    \begin{equation}
        h(y)\leq h(x) + \langle\grad h(x), \Exp^{-1}_{x}(y)\rangle_{x} + \frac{\ell_{h}}{2}\dist(x,y)^2.
    \end{equation}
\end{definition}

To proceed to the bilevel hypergradient estimation, we need the notion of Riemannian Hessian and Riemannian cross-derivatives (Jacobians) (see \cite{han2023riemannian}).

\begin{definition}[Riemannian Hessian]
    For function $f:\M\rightarrow \RR$, the Riemannian Hessian is a symmetric 2-form $H(f): \T\M\times \T\M\rightarrow \RR$ defined as: $\forall \xi, \eta\in \T\M$,
    $$
        H(f)(\xi, \eta) = \langle \nabla_\xi\grad f, \eta \rangle,
    $$
    where $\nabla$ here is the Levi-Civita connection (see \cite{lee2006riemannian}). $H$ can also be interpreted as a linear map $H(f): \T\M\rightarrow \T\M$, $\forall \xi\in \T_{x}\M$,
    $$
        H(f)(\xi) = \nabla_\xi\grad f.
    $$
\end{definition}

\begin{definition}[Riemannian cross-derivatives]
    For a smooth function defined on product manifold $f:\M\times\N\rightarrow \RR$, the Riemannian cross-derivatives is defined as a linear mapping $\grad_{x, y}^2(f):\T\M\rightarrow \T\N$ such that $\forall \xi\in \T_{x}\M$,
    $$
        \grad_{x, y}^2(f)[\xi] = \D_{x}\grad_y f(x, y) [\xi],
    $$
    where $\D_x$ is the differential with respect to variable $x$. $\grad_{y, x}^2(f)$ is defined similarly.
\end{definition}

A useful fact is that $\grad_{x, y}^2(f)$ and $\grad_{y, x}^2(f)$ are adjoint operators. 
\begin{proposition}\label{prop_cross_derivative}
    $\mathrm{grad}_{x,y}^2$ and $\mathrm{grad}_{y,x}^2$ are adjoints, i.e. 
    $$
    \langle \eta,\mathrm{grad}_{x,y}^2f(x, y)[\xi] \rangle_{y} = \langle\mathrm{grad}_{y,x}^2f(x,y)[\eta],\xi \rangle_{x}, \forall \xi\in \T_{x}\M\text{ and }\forall \eta\in \T_{y}\N,
    $$
    where $f\in\mathcal{C}^1(\M)$ is any continuously differentiable function over $\M$.
\end{proposition}
\begin{proof}
    Note 
    $$
        \langle \eta,\mathrm{grad}_{x,y}^2f(x, y)[\xi] \rangle_{y}=\xi( \langle \eta, \mathrm{grad}_y f(x,y) \rangle_y) = \xi(\eta(f)),
    $$
    and similarly
    $$
        \langle\mathrm{grad}_{y,x}^2f(x,y)[\eta],\xi \rangle_{x}=\eta(\xi(f)).
    $$

    Note that here $\xi$ and $\eta$ are actually acting on different coordinates of $f$. We can extend $\Tilde{\xi}(x,y)=(\xi(x), 0)\in \T_x\M\times\T_y\N$ and similarly $\Tilde{\eta}(x,y)=(0, \eta(y))\in \T_x\M\times\T_y\N$. Now subtracting the above equations we have 
    $$
    \langle \eta,\mathrm{grad}_{x,y}^2 f(x,y)[\xi] \rangle_{y} -\langle\mathrm{grad}_{y,x}^2f(x,y)[\eta],\xi \rangle_{x} = [\Tilde{\xi}, \Tilde{\eta}](f),
    $$
    where $[\Tilde{\xi}, \Tilde{\eta}] = \Tilde{\xi}\Tilde{\eta} - \Tilde{\eta}\Tilde{\xi}$ is the Lie bracket. It is easy to verify in local coordinates that $[\Tilde{\xi}, \Tilde{\eta}]$ is zero since $\Tilde{\xi}$ and $\Tilde{\eta}$ act on disjoint local coordinates.
\end{proof}

\section{Bilevel hypergradient estimation on Riemannian manifolds}

We first inspect the calculation of the hypergradient for problem (\ref{deterministic_problem}), namely the Riemannian gradient $\grad \Phi(x)$. Notice that $y^*$ is actually a map $\M\rightarrow \N$, thus we need to follow the notion of the differential of maps between manifolds. We can calculate the Riemannian gradient $\grad \Phi(x)$ as follows.
\begin{proposition}
    The Riemannian gradient $\grad \Phi(x)$ is given by:
    \begin{equation}\label{grad_Phi}
    \grad \Phi(x) = \grad_x f(x,y^*(x)) - \grad_{y, x}^2 g(x,y^*(x))[v^*(x)],
    \end{equation}
    where $v^*(x)\in T_{y^*(x)}\N$ is the solution of the following equation:
    \begin{equation}\label{AID_subproblem}
        H_y(g(x,y^*(x))) (v) = \grad_{y}f(x,y^*(x)),
    \end{equation}
    where $H_y$ is the Riemannian Hessian for the $y$ variable.
\end{proposition}
\begin{proof}
    By chain rule,
    \begin{equation}
        \d \Phi(x) =\d_x f(x,y^*(x)) + \d_{y}f(x,y^*(x))\circ (\D y^*(x)),
    \end{equation}
    where $\d$ and $\D$ represent the differential operators. Notice that the above equation holds in the cotangent space. Since the Riemannian gradients are defined as $\grad \Phi(x)\in \T_{x}\M$ s.t. $\forall\xi\in \T_{x}\M$, $\operatorname{d}\Phi(x)(\xi)=\langle\grad\Phi(x), \xi\rangle$, we get from the above equality that
    \begin{equation}\label{temp0}
        \langle\grad \Phi(x), \xi\rangle =\langle\grad_x f(x,y^*(x)), \xi\rangle +\langle\grad_{y}f(x,y^*(x)), \D y^*(x)(\xi)\rangle.
    \end{equation}
    
    Now we have the following optimality condition from the $y$ lower-level problem:
    $$
        \grad_{y}g(x,y^*(x))=0.
    $$
    By taking the differential for $x$ on both sides of the above formula we get: $\forall \xi\in \T_{x}\M$,
    \begin{equation}\label{lower_optimality}
        \grad_{x, y}^2 g(x,y^*(x))(\xi) + H_y(g(x,y^*(x)))(\D y^*(x)(\xi))=0.
    \end{equation}
    Now taking the inner-product of both sides of the above equation with $v^*(x)$, we get
    $$
        \langle v^*(x), \grad_{x, y}^2 g(x,y^*(x))(\xi) \rangle + \langle  \grad_y f(x,y^*(x)), \D y^*(x)(\xi)\rangle = 0.
    $$
    Therefore we get the final result by plugging back the above equation to \eqref{temp0} and applying Proposition \ref{prop_cross_derivative}.
\end{proof}

When both $\M$ and $\N$ are embedded submanifolds (of two different Euclidean spaces $\RR^{M}$ and $\RR^{N}$), and $\bar{f}:\RR^{M}\times\RR^{N}\rightarrow\RR$ which restricts to $f:\M\times\N\rightarrow\RR$ naturally. The Riemannian gradients of $f$ are simply projections of the Euclidean gradients onto the tangent spaces:
\begin{equation}
    \grad_x f(x,y)=\proj_{\T_x\M}(\nabla_x \bar{f}(x, y)),\ \grad_y f(x,y)=\proj_{\T_y\N}(\nabla_y \bar{f}(x, y)),
\end{equation}
and the cross-derivatives are calculated as follows as a matrix:
\begin{equation}
    \grad_{x, y}^2 f(x,y) = P_y (\nabla_{x, y}^2 \bar{f}(x, y)) P_x,
\end{equation}
where $P_x=\proj_{\T_x\M}\in\RR^{M\times M}$ and $P_y=\proj_{\T_y\N}\in\RR^{N\times N}$ are projection matrices onto tangent spaces, and $\nabla_{x, y}^2 \bar{f}(x, y)\in \RR^{N\times M}$ is the regular partial gradient, namely $[\nabla_{x, y}^2 \bar{f}(x, y)]_{j, i} = \frac{\partial^2 \bar{f}}{\partial y_j\partial x_i}$.

In practice we cannot solve the inner minimization and (\ref{AID_subproblem}) exactly. Suppose we have a point $y\in\N$, we can solve the approximate problem of \eqref{AID_subproblem}:
\begin{equation}\label{AID_subproblem_approx}
    H_y(g(x,y)) [v] = \grad_{y}f(x,y)
\end{equation}
with an $N$-step conjugate gradient method, yielding $\hat{v}^N(x, y)$, then we can estimate
\begin{equation}
    \grad\Phi(x) \approx \grad_x f(x,y) - \grad_{y, x}^2 g(x,y)[\hat{v}^N(x, y)],
\end{equation}
which we further refer to as the approximate implicit differentiation (AID) estimate of \eqref{grad_Phi}. For the rest of the paper we denote $h_{g}^{k, t}:=\grad_y g(x^k, y^{k,t})$ and
\begin{equation}\label{approximate_AID}
    h_{\Phi}(x, y):=\grad_x f(x,y) - \grad_{y, x}^2 g(x,y)[\hat{v}^N(x, y)].
\end{equation}
We abbreviate the notation by $h_{\Phi}^k:=h_{\Phi}(x^k, y^k)$ in the algorithms.

For the stochastic problem (\ref{stochastic_problem}), we have an estimate of the gradient of the stochastic function $G(x,y;\zeta)$ described as follows (see~\cite{hong2020two,chen2021tighter}). We first update the inner problem $y^{t}\leftarrow \Exp_{y^{t-1}}(-\alpha\grad_{y}G(x,y^{t-1};\zeta^{t-1}))$ for $t=1,...,T$. Meanwhile the estimate for the gradient of $F$ and the second-order gradient of $G$ will require us to further have independent samples $\xi$ and $\zeta_{(q)}$, $q=0,1,...,Q$, so that we define the stochastic gradient estimator as:
\begin{equation}\label{stoch_grad_estimate}
    \grad \Phi\left(x\right)\approx\grad_x F(x,y^T;\xi) - \grad_{y, x}^2 G(x,y^T;\zeta_{0})[v_{Q}(x, y^T)],
\end{equation}
where $v_{Q}$ is the approximation of \eqref{AID_subproblem}, defined as (see \cite[Lemma 1]{hong2020two}):
\begin{equation}\label{vq}
    v_{Q}(x, y):=\eta Q\prod_{q=1}^{Q'}(I - \eta H_y( G(x, y;\zeta_{(q)})) ) [\grad_{y} F\left(x, y ; \xi\right)],
\end{equation}
where $Q'$ is drawn uniformly from $\{0,1,...,Q-1\}$ and the extra parameter $\eta$ will be later determined to ensure a better approximation to \eqref{AID_subproblem}, motivated by the Neumann series $\sum_{i=0}^{\infty}U^i = (I-U)^{-1}$.

From now on we denote $\Tilde{h}_{g}^{k, t}=\grad_{y}G(x^k,y^{k, t};\zeta_{k, t})$ and
\begin{equation}\label{approximate_Neumann_stochastic}
    \Tilde{h}_{\Phi}^k := \grad_x F(x^k,y^k;\xi_k) - \grad_{y, x}^2 G(x^k,y^k;\zeta_{k, (0)})[v_{Q}^k],
\end{equation}
where
\begin{equation*}
    v_{Q}^k:=\eta Q\prod_{q=1}^{Q'}(I - \eta H_y( G(x^k, y^k;\zeta_{k, (q)})) ) [\grad_{y} F (x^k, y^k ; \xi_k)].
\end{equation*}

\section{Deterministic Algorithm RieBO and Its Convergence}

We propose RieBO (Algorithm \ref{algo_bilevel_AID_ITD}) for the deterministic bilevel manifold optimization \eqref{deterministic_problem}. The algorithm is a generalization of its Euclidean counterpart proposed in \cite{ji2021bilevel} where we employ conjugate gradient method to solve the hypergradient estimation problem \eqref{approximate_AID}. 

For the deterministic case, we utilize the following notion of stationarity:
\begin{definition}\label{def_deterministic_eps_stationary}
    A point $x\in\M$ is called an $\epsilon$-stationary point for \eqref{deterministic_problem} if $\|\nabla\Phi(x)\|^2\leq\epsilon$.
\end{definition}

We need the following assumptions for the convergence analysis.
\begin{assumption}\label{assumption_1}
    The manifolds $\M$ and $\N$ are complete Riemannian manifolds. Moreover, $\N$ is a Hadamard manifold whose sectional curvature is lower bounded by $\iota < 0$ \footnote{We make this assumption so that the lower function $g$ can be geodesically strongly convex, see \cite{zhang2016first}.}. 
\end{assumption}
We use the notation $\langle\cdot, \cdot\rangle_x$ and $\langle\cdot, \cdot\rangle_y$ to represent their Riemannian metrics, for $x\in\M$ and $y\in\N$. The corresponding norms are $\|\cdot\|_x$ and $\|\cdot\|_y$. Note that from now on we may omit the subscript since the corresponding manifold and tangent space can be identified by the vector in the ``$\cdot$" position.

Following \cite{zhang2016first}, it is very important that the quantity $\tau(\iota, \dist(y^{k, t}, y^*(x^k)))$ is bounded during the lower level update, where 
\begin{equation}
    \tau(\iota, c):=\frac{\sqrt{|\iota|c}}{\tanh(\sqrt{|\iota|c})}.
\end{equation}
Therefore we need the following assumption.

\begin{assumption}\label{assumption_2}
    The lower level objective function $g(x,y)$ is $\mu$-geodesically strongly convex with respect to $y$. Note that the total objective $\Phi(x)=f(x,y^*(x))$ may still be nonconvex. Moreover, we assume that the quantity $\tau(\iota, \dist(y^{k, t}, y^*(x^k)))$ is always upper bounded by $\tau$ for all $k$ and $t$ \footnote{Note that this assumption is satisfied if the lower level problem is conducted in a compact subset in $\N$ and assuming that the iterates of the algorithm stay in this compact region, see \cite{zhang2016first}.}.
\end{assumption}

\begin{assumption}[Smoothness]\label{assumption_3}
    For simplicity denote $z=(x,y)$ and $z'=(x',y')$, also denote $\dist(z,z')=\sqrt{\dist(x,x')^2 + \dist(y,y')^2}$ (note that these two distances are on different manifolds). Moreover, $f$ and $g$ satisfy the following assumptions. 
    \begin{itemize}
        \item $f$ satisfies $\ell_{f, 0}$-Lipschitzness:
        $$
            f(z) - f(z')\leq \ell_{f, 0}\dist(z,z').
        $$
        \item $\grad f=[\grad_{x}f, \grad_{y} f]$ and $\grad g=[\grad_{x}g, \grad_{y} g]$ are $\ell_{f, 1}$ and $\ell_{g, 1}$Lipschitz, i.e.,
        $$
        \begin{aligned}
            &\|\grad f(z) - P_{z'\rightarrow z}\grad f(z')\|\leq \ell_{f, 1}\dist(z,z') \\
            &\|\grad g(z) - P_{z'\rightarrow z}\grad g(z')\|\leq \ell_{g, 1}\dist(z,z'),
        \end{aligned}
        $$
        where $P_{z'\rightarrow z}$ is the parallel transport on the product manifold $\M\times\N$ and the norm on the left hand side is also induced by the product Riemannian metric on $\M\times \N$ \footnote{It is worth noticing that in our convergence result, we need $\tau<\ell_{g, 1}/2$. We comment here that this assumption is not a big issue since if $g$ is Lipschitz smooth with parameter $\ell_{g, 1}$, then it is also Lipschitz smooth with any parameters $\ell > \ell_{g, 1}$. We could always pick up a parameter that satisfies $\tau<\ell_{g, 1}/2$, with some sacrifice of the convergence speed. Note that in the Euclidean case $\tau=0$ so $\tau<\ell_{g, 1}/2$ holds naturally.}.
    \end{itemize}
\end{assumption}

\begin{assumption}[Hessian Smoothness]\label{assumption_4}
    The second-order derivatives $\grad_{x,y}^2g(z)$ and $H_{y}(g(z))$ are $\ell_{g, 2}$ Lipschitz (we use the same constant here for simplicity), i.e. for $z=(x,y)$ and $z'=(x',y')$, we have
    $$
    \begin{aligned}
        &\left\|\grad_{x,y}^2g(z) - P_{y^*(x')\rightarrow y^*(x)}\circ\grad_{x,y}^2g(z') \circ P_{x'\rightarrow x}\right\|_{\mathrm{op}}\leq \ell_{g, 2}\dist(z,z') \\
        &\left\|H_{y}(g(z)) - P_{y^*(x')\rightarrow y^*(x)}\circ H_{y}(g(z')) \circ P_{y^*(x)\rightarrow y^*(x')}\right\|_{\mathrm{op}}\leq \ell_{g, 2}\dist(z,z').
    \end{aligned}
    $$
    Here the norms on the left hand side are the operator norms. We will keep the subscript $\mathrm{op}$ whenever it comes to the operator norm, in order to distinguish it from the norms on the tangent spaces.
\end{assumption}

We have the following smoothness lemma under the above assumptions.
\begin{lemma}\label{lemma0}
    Suppose Assumptions \ref{assumption_1}, \ref{assumption_2}, \ref{assumption_3} and \ref{assumption_4} hold, then functions $y^*(x)$ and $\Phi(x):=f(x,y^*(x))$ satisfy: $\forall x,x'\in\M$
    \begin{subequations}\label{phi_smoothness}
    \begin{align}
        &\dist(y^*(x),y^*(x'))\leq \kappa \dist(x,x'),\ \kappa = \frac{\ell_{g, 1}}{\mu}\label{phi_smoothness-1}\\
        &\|\D y^*(x)-P_{y^*(x')\rightarrow y^*(x)}\circ \D y^*(x')\circ P_{x'\rightarrow x}\|_{\mathrm{op}}\leq L_{y^*}\dist(x,x')\label{phi_smoothness-2}\\
        &\|\grad \Phi(x) - P_{x'\rightarrow x}\grad \Phi(x')\|\leq L_{\Phi}\dist(x,x'),\label{phi_smoothness-3}
    \end{align}
    \end{subequations}
    where
    \begin{equation}
        L_{y^*}:= \bigg(1 + \frac{\ell_{g, 2}}{\mu}\bigg)\frac{\ell_{g, 2}}{\mu}\sqrt{1+\kappa^2}=\mathcal{O}(\kappa^2),
    \end{equation}
    and
    \begin{equation}
        L_{\Phi}:= \ell_{f, 1} \sqrt{1+\kappa^2} + \ell_{g, 2} \frac{\ell_{f, 0}}{\mu}  + \ell_{g, 1}\bigg(\frac{\ell_{f, 0}\ell_{g, 2}}{\mu^2}\sqrt{1+\kappa^2} + \frac{\ell_{f, 1}}{\mu}\bigg)=\mathcal{O}(\kappa^3).
    \end{equation}
\end{lemma}

\begin{proof}
    For \eqref{phi_smoothness-1}, by \eqref{lower_optimality} and our Assumptions \ref{assumption_2}, \ref{assumption_3} and \ref{assumption_4}, we have
    \[
    \begin{split}
        \|\D y^*(x)\|_{\mathrm{op}}=\|(H_{y}(g(x,y^*(x))))^{-1}\circ\grad_{x,y}^2g(x,y^*(x))\|_{\mathrm{op}}\leq\frac{\ell_{g, 1}}{\mu}.
    \end{split}
    \]
    We thus obtain \eqref{phi_smoothness-1} by a mean value theorem argument in local coordinates (see \href{https://math.stackexchange.com/questions/1450725/lipschitz-maps-between-riemannian-manifolds/1450813#1450813}{this link} for a detailed proof).

    For \eqref{phi_smoothness-2}, we have
    \[
    \begin{split}
        &\|\D y^*(x)-P_{y^*(x')\rightarrow y^*(x)}\circ \D y^*(x')\circ P_{x'\rightarrow x}\|\\
        =&\|(H_{y}(g(x,y^*(x))))^{-1}\circ\grad_{x,y}^2g(x,y^*(x))-P_{y^*(x')\rightarrow y^*(x)}\circ (H_{y}(g(x',y^*(x'))))^{-1}\circ\grad_{x,y}^2g(x',y^*(x'))\circ P_{x'\rightarrow x}\| \\
        \leq & \|(H_{y}(g(x,y^*(x))))^{-1}\circ\grad_{x,y}^2g(x,y^*(x))\\&\qquad\qquad\qquad\qquad\qquad - (H_{y}(g(x,y^*(x))))^{-1} \circ P_{y^*(x')\rightarrow y^*(x)}\circ\grad_{x,y}^2g(x',y^*(x'))\circ P_{x'\rightarrow x}\| \\
        &+ \|(H_{y}(g(x,y^*(x))))^{-1} \circ P_{y^*(x')\rightarrow y^*(x)}\circ\grad_{x,y}^2g(x',y^*(x'))\\&\qquad\qquad\qquad\qquad\qquad -P_{y^*(x')\rightarrow y^*(x)}\circ (H_{y}(g(x',y^*(x'))))^{-1}\circ\grad_{x,y}^2g(x',y^*(x'))\|\\
        \leq & \|(H_{y}(g(x,y^*(x))))^{-1}\| \|\grad_{x,y}^2g(x,y^*(x)) - P_{y^*(x')\rightarrow y^*(x)}\circ\grad_{x,y}^2g(x',y^*(x'))\circ P_{x'\rightarrow x}\| \\
        &+ \|(H_{y}(g(x,y^*(x))))^{-1}  -P_{y^*(x')\rightarrow y^*(x)}\circ (H_{y}(g(x',y^*(x'))))^{-1}\circ P_{y^*(x)\rightarrow y^*(x')}\| \|\grad_{x,y}^2g(x',y^*(x'))\| \\
        \leq & \frac{\ell_{g, 2}}{\mu} \sqrt{\dist(x, x')^2+\dist(y^*(x), y^*(x'))^2}\\& + \ell_{g, 1} \|(H_{y}(g(x,y^*(x))))^{-1} - P_{y^*(x')\rightarrow y^*(x)}\circ (H_{y}(g(x',y^*(x'))))^{-1}\circ P_{y^*(x)\rightarrow y^*(x')}\| \\
        \leq & \frac{\ell_{g, 2}}{\mu} \sqrt{1+\kappa^2}\dist(x, x')+ \ell_{g, 1} \|(H_{y}(g(x,y^*(x))))^{-1} - P_{y^*(x')\rightarrow y^*(x)}\circ (H_{y}(g(x',y^*(x'))))^{-1}\circ P_{y^*(x)\rightarrow y^*(x')}\|_{\mathrm{op}}
    \end{split}
    \]
    where in the second last inequality we used Assumptions \ref{assumption_2}, \ref{assumption_3} and \ref{assumption_4}, and we used \eqref{phi_smoothness-1} for the last inequality. Denote $H_1 = H_{y}(g(x,y^*(x)))$, $P=P_{y^*(x')\rightarrow y^*(x)}$ so that $P_{y^*(x)\rightarrow y^*(x')}=P^{-1}$ and $H_2=H_{y}(g(x',y^*(x')))$, then the last term in the above formula becomes:
    \begin{align}\label{temp3}
    \begin{aligned}
        &\|H_1^{-1} - P H_2^{-1} P\|_{\mathrm{op}}
        = \|H_1^{-1} P^{-1}(H_2 - P^{-1} H_1 P) H_2^{-1} P^{-1}\|_{\mathrm{op}}\\
        \leq & \frac{1}{\mu^2}\|H_2 - P^{-1} H_1 P\|_{\mathrm{op}}
        \leq \frac{\ell_{g, 2}}{\mu^2} \sqrt{\dist(x, x')^2+\dist(y^*(x), y^*(x'))^2}\\
        \leq & \frac{\ell_{g, 2}}{\mu^2}\sqrt{1+\kappa^2}\dist(x, x').
    \end{aligned}
    \end{align}
    Here we used Assumptions \ref{assumption_2}, \ref{assumption_4}, \eqref{phi_smoothness-1} and the fact that the parallel transport $P=P_{y^*(x')\rightarrow y^*(x)}$ is an isometry, i.e., $\|P\|_{\mathrm{op}}=1$. Plugging this back we get
    \begin{align*}
        \|\D y^*(x)-P_{y^*(x')\rightarrow y^*(x)}\circ \D y^*(x')\circ P_{x'\rightarrow x}\|_{\mathrm{op}}
        \leq \bigg(\frac{\ell_{g, 2}}{\mu} + \frac{\ell_{g, 1} \ell_{g, 2}}{\mu^2}\bigg)\sqrt{1+\kappa^2}\dist(x, x'),
    \end{align*}
    which gives \eqref{phi_smoothness-2}.
    
    We now show \eqref{phi_smoothness-3}. Using \eqref{grad_Phi}, we get
    \begin{align}\label{temp4}
    \begin{aligned}
        &\|\grad \Phi(x) - P_{x'\rightarrow x}\grad \Phi(x')\| \leq \|\grad_x f(x,y^*(x)) - P_{x'\rightarrow x}\grad_x f(x',y^*(x'))\|\\
        &+ \| \grad_{y, x}^2 g(x,y^*(x))[v^*(x)] - P_{x'\rightarrow x}\grad_{y, x}^2 g(x',y^*(x'))[v^*(x')] \|.
    \end{aligned}
    \end{align}
    
    For the first term on the right hand side of \eqref{temp4}, by Assumption \ref{assumption_3} and \eqref{phi_smoothness-1}, we have
    \[
    \begin{split}
        &\|\grad_x f(x,y^*(x)) - P_{x'\rightarrow x}\grad_x f(x',y^*(x'))\|\\\leq& \ell_{f, 1} \sqrt{\dist(x, x')^2+\dist(y^*(x), y^*(x'))^2}
        \leq \ell_{f, 1} \sqrt{1+\kappa^2}\dist(x, x').
    \end{split}
    \]
    
    For the second term on the right hand side of \eqref{temp4}, we have
    \[
    \begin{split}
        & \| \grad_{y, x}^2 g(x,y^*(x))[v^*(x)] - P_{x'\rightarrow x}\grad_{y, x}^2 g(x',y^*(x'))[v^*(x')] \| \\
        = & \| \grad_{y, x}^2 g(x,y^*(x))[v^*(x)] - P_{x'\rightarrow x}\grad_{y, x}^2 g(x',y^*(x'))\bigg[ P_{x\rightarrow x'} P_{x'\rightarrow x}[v^*(x')]\bigg] \|\\
        \leq & \| \grad_{y, x}^2 g(x,y^*(x))[v^*(x)] - P_{x'\rightarrow x}\grad_{y, x}^2 g(x',y^*(x'))P_{x\rightarrow x'} [v^*(x)] \| \\
        &+ \| P_{x'\rightarrow x}\grad_{y, x}^2 g(x',y^*(x'))P_{x\rightarrow x'} [v^*(x)] - P_{x'\rightarrow x}\grad_{y, x}^2 g(x',y^*(x'))P_{x\rightarrow x'} P_{x'\rightarrow x}[v^*(x')]\| \\
        \leq & \| \grad_{y, x}^2 g(x,y^*(x))- P_{x'\rightarrow x}\grad_{y, x}^2 g(x',y^*(x'))P_{x\rightarrow x'} \|_{\mathrm{op}} \|v^*(x)\| \\
        &+ \|\grad_{y, x}^2 g(x',y^*(x'))\|_{\mathrm{op}}\| v^*(x) - P_{x'\rightarrow x}v^*(x')\|.
    \end{split}
    \]
    Since $v^*(x)$ is the solution of \eqref{AID_subproblem}, also by Assumptions \ref{assumption_2} and \ref{assumption_3}, we have $\|v^*(x)\|\leq \ell_{f, 0}/\mu$, also
    \begin{equation}\label{eq_bound_v_star}
    \begin{split}
    &\| v^*(x) - P_{x'\rightarrow x}v^*(x')\|\\
    =&\| (H_y(g(x,y^*(x))))^{-1}\grad_{y}f(x,y^*(x)) - P_{x'\rightarrow x}(H_y(g(x',y^*(x'))))^{-1}\grad_{y}f(x',y^*(x'))\| \\
    =&\| (H_y(g(x,y^*(x))))^{-1}\grad_{y}f(x,y^*(x)) - P_{x'\rightarrow x}(H_y(g(x',y^*(x'))))^{-1}P_{x\rightarrow x'}P_{x'\rightarrow x}\grad_{y}f(x',y^*(x'))\| \\
    \leq & \| (H_y(g(x,y^*(x))))^{-1} - P_{x'\rightarrow x}(H_y(g(x',y^*(x'))))^{-1}P_{x\rightarrow x'}\|_{\mathrm{op}}\|\grad_{y}f(x,y^*(x))\| \\
    &+ \|(H_y(g(x',y^*(x'))))^{-1}\|_{\mathrm{op}}\| \grad_{y}f(x,y^*(x)) - P_{x'\rightarrow x}\grad_{y}f(x',y^*(x'))\| \\
    \leq & \bigg(\frac{\ell_{f, 0}\ell_{g, 2}}{\mu^2}\sqrt{1+\kappa^2} + \frac{\ell_{f, 1}}{\mu}\bigg)\dist(x, x'),
    \end{split}
    \end{equation}
    where we used \eqref{temp3}, Assumptions \ref{assumption_2} and \ref{assumption_3}.
    
    Combining the above bounds and plugging it to \eqref{temp4}, we get
    \begin{align*}
    \begin{aligned}
        &\|\grad \Phi(x) - P_{x'\rightarrow x}\grad \Phi(x')\|\\
        \leq & \ell_{f, 1} \sqrt{1+\kappa^2}\dist(x, x') + \ell_{g, 2} \frac{\ell_{f, 0}}{\mu} \dist(x, x') + \ell_{g, 1}\bigg(\frac{\ell_{f, 0}\ell_{g, 2}}{\mu^2}\sqrt{1+\kappa^2} + \frac{\ell_{f, 1}}{\mu}\bigg)\dist(x, x'),
    \end{aligned}
    \end{align*}
    which proves \eqref{phi_smoothness-3}.
\end{proof}

\begin{algorithm}[!ht]
\SetKwInOut{Input}{input}
\SetKwInOut{Output}{output}
\SetAlgoLined
\Input{$K$, $T$, $N$(steps for conjugate gradient), stepsize $\{\alpha_k,\beta_k\}$, initializations $x^0\in\M, y^0\in\N$}
    \For{$k=0,1,2,...,K-1$}{
        Set $y^{k, 0}=y^{k-1}$\;
        \For{$t=0,...,T-1$}{
            Update $y^{k, t+1}\leftarrow \Exp_{y^{k, t}}(-\beta_k h_{g}^{k, t})$ with $h_{g}^{k, t}:=\grad_y g(x^k, y^{k,t})$ \;
        }
        Set $y^{k}\leftarrow y^{k, T}$\;
        
        Update $x^{k+1}\leftarrow\Exp_{x^{k}}(-\alpha_k h_{\Phi}^k)$ as in \eqref{approximate_AID}, where $\hat{v}^N(x^k, y^k)$ is given by an $N$-step conjugate gradient update, with $\hat{v}^0(x^k, y^k)=P_{y^{k-1}\rightarrow y^k}\hat{v}^N(x^{k-1}, y^{k-1})$\;
    }
 \label{algo_bilevel_AID_ITD}
 \caption{Algorithm for {\bf Rie}mannian (deterministic) {\bf B}ilevel {\bf O}ptimization (\bf{RieBO})}
\end{algorithm}


\begin{theorem}\label{theorem1}
    Suppose Assumptions \ref{assumption_1}, \ref{assumption_2}, \ref{assumption_3} and \ref{assumption_4} hold, and take the parameters $\beta_k=\beta\leq\frac{1}{\ell_{g, 1}}$, $\alpha_k=\alpha\leq\frac{1}{8 L_{\Phi}}$, $T\geq\mathcal{O}(\kappa)$ and conjugate gradient iteration number $N\geq\mathcal{O}(\sqrt{\kappa})$. Then RieBO (Algorithm \ref{algo_bilevel_AID_ITD}) satisfies:
    \begin{equation}
        \frac{1}{K}\sum_{k=0}^{K-1}\|\grad \Phi(x^{k})\|^2\leq \mathcal{O}\left(\frac{L_{\Phi}}{K}\right).
    \end{equation}
    The specific choice parameters are given in the proof for the simplicity of the statement. In order to achieve an $\epsilon$-accurate stationary point, the complexity is given by:
    \begin{itemize}
        \item Gradients: $\operatorname{Gc}(f,\epsilon)=\mathcal{O}(\kappa^3\epsilon^{-1})$, $\operatorname{Gc}(g,\epsilon)=\mathcal{O}(\kappa^4\epsilon^{-1})$;
        
        \item Jacobian and Hessian-vector products: $\operatorname{JV}(g, \epsilon)=\mathcal{O}(\kappa^3\epsilon^{-1})$, $\operatorname{HV}(g, \epsilon)=\mathcal{O}(\kappa^{3.5}\epsilon^{-1})$.
    \end{itemize}
\end{theorem}

To prove this theorem, we need the following lemmas. The first lemma quantifies the error when optimizing \eqref{AID_subproblem_approx} with $N$-step conjugate gradient method, see \citet[Equation (17)]{grazzi2020iteration}\footnote{Note that here the Hessian matrix $H_y(g(x, y))$ is full-rank in the tangent space. However if it is an embedded submanifold then $H_y(g(x, y))$ is actually rank-deficient matrix in the ambient Euclidean space. This is not a concern for showing the linear rate of convergence since we can always conduct CG steps only on the tangent spaces (as Euclidean subspaces of the ambient Euclidean space). It is known that the convergence is still linear even if $H_y(g(x, y))$ is rank-deficient, see \cite{hayami2018convergence} for a detailed inspection.}.
\begin{lemma}\label{lemma_cg_error}
    Suppose we solve \eqref{AID_subproblem_approx} with $N$-step conjugate gradient method with the initial point $\hat{v}^0(x, y)$ and output $\hat{v}^N(x, y)$, then we have
    $$
    \|\hat{v}^N(x, y)-\tilde{v}\| \leq \sqrt{\kappa}\left(\frac{\sqrt{\kappa} - 1}{\sqrt{\kappa} + 1}\right)^N \|\hat{v}^0(x, y)-\tilde{v}\|,
    $$
    where $\tilde{v}$ is the exact solution of \eqref{AID_subproblem_approx}.
\end{lemma}

The next lemma quantifies the error of the inner loop, i.e. the $T$ steps where we do Riemannian gradient descent for the lower problem in RieBO (Algorithm \ref{algo_bilevel_AID_ITD}).
\begin{lemma}\label{lemma_inner}
    Suppose Assumptions \ref{assumption_1}, \ref{assumption_2}, \ref{assumption_3} and \ref{assumption_4} hold, and we take $\beta_k=\beta= 1/\ell_{g, 1}$ as a constant, then RieBO satisfies:
    \begin{equation}
        \dist(y^{k, T}, y^*(x^k))^2\leq (1 - 2\mu\tau\beta^2)^T \dist(y^{k, 0}, y^*(x^k))^2.
    \end{equation}
\end{lemma}
\begin{proof}
    For simplicity, we denote $h(y) = g(x^k, y)$, so that $y^*(x^k)$ is the optimal solution of $h$. We also omit $k$ in this proof, i.e., the update becomes:
    $$
    y^{t+1}\leftarrow \Exp_{y^{t}}(-\beta_k \grad h(y^t)).
    $$

    By the notion of geodesic smoothness and geodesic convexity, we have
    \[
    \begin{split}
        &h(y^{t+1}) - h(y^*) = h(y^{t+1}) - h(y^t) + h(y^{t}) - h(y^*) \\
        \leq & \langle\grad h(y^t), \Exp_{y^t}(y^{t+1})\rangle + \frac{\ell_{g, 1}}{2}\|\Exp_{y^t}(y^{t+1})\|^2 - \langle \grad h(y^t), \Exp_{y^{t}}(y^*) \rangle - \frac{\mu}{2}\dist(y^{t}, y^*)^2 \\
        = & -(\beta - \frac{\beta^2\ell_{g, 1}}{2})\|\grad h(y^t)\|^2 - \langle \grad h(y^t), \Exp_{y^{t}}(y^*) \rangle - \frac{\mu}{2}\dist(y^{t}, y^*)^2,
    \end{split}
    \]
    i.e.,
    \begin{equation}\label{inner_temp1}
        (\beta - \frac{\beta^2\ell_{g, 1}}{2})\|\grad h(y^t)\|^2 - \langle \grad h(y^t), \Exp_{y^{t}}(y^*) \rangle \leq - \frac{\mu}{2}\dist(y^{t}, y^*)^2.
    \end{equation}
    Now by \citet[Corollary 8]{zhang2016first}, we have,
    \[
    \begin{split}
        \dist(y^{t+1}, y^*)^2 \leq & \dist(y^{t}, y^*)^2 + 2\beta\langle \grad h(y^t), \Exp_{y^t}(y^*)\rangle + \tau\beta^2\|\grad h(y^t)\|^2 \\
        \leq & (1 - 2\mu\tau\beta^2 )\dist(y^{t}, y^*)^2, 
    \end{split}
    \]
    where the last inequality is by \eqref{inner_temp1} and $\beta=1/\ell_{g,1}$. The proof is done by repeatedly applying the above inequality from $t=T-1$ back to $t=0$.
\end{proof}

The next lemma quantifies the error between our estimation $h_{\Phi}^k$ and the true upper level gradient ${\grad}\Phi(x^{k})$.
\begin{lemma}\label{lemma1}
    Suppose Assumptions \ref{assumption_1}, \ref{assumption_2}, \ref{assumption_3} and \ref{assumption_4} hold, then RieBO satisfies:
    \begin{equation}
    \begin{split}
        \|h_{\Phi}^k - {\grad}\Phi(x^{k})\|
        \leq& \Gamma (1-2\mu\tau\beta^2)^{T/2} \dist(y^*(x^k), y^{k-1})\\ &+ \ell_{g, 1}\sqrt{\kappa}\left(\frac{\sqrt{\kappa} - 1}{\sqrt{\kappa} + 1}\right)^N \|\hat{v}^0(x^{k}, y^{k})-P_{y^*(x^k)\rightarrow y^k} v^*(x^k)\|,
    \end{split}
    \end{equation}
    where $h_{\Phi}^k$ is the estimate from \eqref{approximate_AID} and we have the parameters:
    \begin{equation}
    \begin{split}
        &\tilde{v}^{k} = (H_y(g(x^k,y^k)))^{-1}\grad_{y}f(x^k,y^k) \\
        &\Gamma=\ell_{f, 1} + \frac{\ell_{f, 0}\ell_{g, 2}}{\mu} + \ell_{g, 1}\left(1+\sqrt{\kappa}\left(\frac{\sqrt{\kappa} - 1}{\sqrt{\kappa} + 1}\right)^N\right)\bigg(\ell_{g, 2}\ell_{f, 0} + \frac{\ell_{f, 1}}{\mu}\bigg).
    \end{split}
    \end{equation}
\end{lemma}
\begin{proof}
    We first restate the expression \eqref{grad_Phi} for ${\grad}\Phi(x)$ and \eqref{approximate_AID} for $h_{\Phi}^k$:
    $$
    \begin{aligned}
        &\grad\Phi(x^k) = \grad_x f(x^k,y^*(x^k)) - \grad_{y, x}^2 g(x^k,y^*(x^k))[v^*(x^k)], \\
        &h_{\Phi}(x^k, y^k):=\grad_x f(x^k,y^k) - \grad_{y, x}^2 g(x^k,y^k)[\hat{v}^N(x^k, y^k)].
    \end{aligned}
    $$
    Thus, 
    \[
    \begin{split}
        &\|h_{\Phi}^k - {\grad}\Phi(x^{k})\|\leq \|\grad_x f(x^k,y^*(x^k)) - \grad_x f(x^k,y^k)\| \\
        &+  \|\grad_{y, x}^2 g(x^k,y^*(x^k))[v^*(x^k)] - \grad_{y, x}^2 g(x^k,y^k)[\hat{v}^N(x^k, y^k)]\| \\
        \leq & \|\grad_x f(x^k,y^*(x^k)) - \grad_x f(x^k,y^k)\| \\
        &+ \|\grad_{y, x}^2 g(x^k,y^*(x^k))[v^*(x^k)] - \grad_{y, x}^2 g(x^k,y^k)[P_{y^*(x^k)\rightarrow y^k} v^*(x^k)]\| \\
        &+ \|\grad_{y, x}^2 g(x^k,y^k)[P_{y^*(x^k)\rightarrow y^k} v^*(x^k)] - \grad_{y, x}^2 g(x^k,y^k)[\hat{v}^N(x^k, y^k)]\| \\
        \leq & \ell_{f, 1}\dist(y^*(x^k), y^k) \\
        &+ \|\grad_{y, x}^2 g(x^k,y^*(x^k)) - \grad_{y, x}^2 g(x^k,y^k)\circ P_{y^*(x^k)\rightarrow y^k}\|_{\mathrm{op}}\|v^*(x^k)\| \\
        &+ \|\grad_{y, x}^2 g(x^k,y^k)\|_{\mathrm{op}}\|P_{y^*(x^k)\rightarrow y^k} v^*(x^k) - \hat{v}^N(x^k, y^k)\| \\
        \leq & \left(\ell_{f, 1} + \frac{\ell_{f, 0}\ell_{g, 2}}{\mu}\right)\dist(y^*(x^k), y^k) + \ell_{g, 1}\|P_{y^*(x^k)\rightarrow y^k} v^*(x^k) - \hat{v}^N(x^k, y^k)\|.
    \end{split}
    \]
    Following Lemma \ref{lemma_cg_error}, we have
    \[
    \begin{split}
        &\|P_{y^*(x^k)\rightarrow y^k} v^*(x^k) - \hat{v}^N(x^k, y^k)\| \leq  \|P_{y^*(x^k)\rightarrow y^k} v^*(x^k) - \tilde{v}^{k}\| + \|\tilde{v}^{k} - \hat{v}^N(x^k, y^k)\|\\
        \leq & \|P_{y^*(x^k)\rightarrow y^k} v^*(x^k) - \tilde{v}^{k}\| + \sqrt{\kappa}\left(\frac{\sqrt{\kappa} - 1}{\sqrt{\kappa} + 1}\right)^N \|\hat{v}^0(x^{k}, y^{k})-\tilde{v}^{k}\|\\
        \leq & \left(1+\sqrt{\kappa}\left(\frac{\sqrt{\kappa} - 1}{\sqrt{\kappa} + 1}\right)^N\right)\|P_{y^*(x^k)\rightarrow y^k} v^*(x^k) - \tilde{v}^{k}\| + \sqrt{\kappa}\left(\frac{\sqrt{\kappa} - 1}{\sqrt{\kappa} + 1}\right)^N \|\hat{v}^0(x^{k}, y^{k})-P_{y^*(x^k)\rightarrow y^k} v^*(x^k)\|.
    \end{split}
    \]
    For $\|P_{y^*(x^k)\rightarrow y^k} v^*(x^k) - \tilde{v}^{k}\|$, by the definitions of $\tilde{v}_{k}$ and $v_{k}^*$, we have
    \begin{equation}\label{eq_bound_tilde_v_v_star}
    \begin{split}
        &\|P_{y^*(x^k)\rightarrow y^k} v^*(x^k) - \tilde{v}^{k}\|\\ = &\|P_{y^*(x^k)\rightarrow y^k}(H_y(g(x^k,y^*(x^k))))^{-1}\grad_{y}f(x^k,y^*(x^k)) - (H_y(g(x^k,y^k)))^{-1}\grad_{y}f(x^k,y^k)\| \\
        \leq & \|P_{y^*(x^k)\rightarrow y^k}(H_y(g(x^k,y^*(x^k))))^{-1}\grad_{y}f(x^k,y^*(x^k)) - (H_y(g(x^k,y^k)))^{-1}P_{y^*(x^k)\rightarrow y^k}\grad_{y}f(x^k,y^*(x^k))\|\\
        &+ \|(H_y(g(x^k,y^k)))^{-1}P_{y^*(x^k)\rightarrow y^k}\grad_{y}f(x^k,y^*(x^k)) - (H_y(g(x^k,y^k)))^{-1}\grad_{y}f(x^k,y^k)\| \\
        \leq & \|(H_y(g(x^k,y^*(x^k))))^{-1} - P_{y^k\rightarrow y^*(x^k)}(H_y(g(x^k,y^k)))^{-1}P_{y^*(x^k)\rightarrow y^k}\|_{\mathrm{op}}\|\grad_{y}f(x^k,y^*(x^k))\| \\
        &+ \|(H_y(g(x^k,y^k)))^{-1}\|_{\mathrm{op}}\|P_{y^*(x^k)\rightarrow y^k}\grad_{y}f(x^k,y^*(x^k)) - \grad_{y}f(x^k,y^k)\| \\
        \leq & \bigg(\ell_{g, 2}\ell_{f, 0} + \frac{\ell_{f, 1}}{\mu}\bigg)\dist(y^k, y^*(x^k)).
    \end{split}
    \end{equation}
    Therefore, we get
    \[
    \begin{split}
        &\|h_{\Phi}^k - {\grad}\Phi(x^{k})\|\leq \left(\ell_{f, 1} + \frac{\ell_{f, 0}\ell_{g, 2}}{\mu}\right)\dist(y^*(x^k), y^k)\\& + \ell_{g, 1}\left(1+\sqrt{\kappa}\left(\frac{\sqrt{\kappa} - 1}{\sqrt{\kappa} + 1}\right)^N\right)\|P_{y^*(x^k)\rightarrow y^k} v^*(x^k) - \tilde{v}^{k}\| \\& + \ell_{g, 1}\sqrt{\kappa}\left(\frac{\sqrt{\kappa} - 1}{\sqrt{\kappa} + 1}\right)^N \|\hat{v}^0(x^{k}, y^{k})-P_{y^*(x^k)\rightarrow y^k} v^*(x^k)\| \\
        \leq& \bigg(\ell_{f, 1} + \frac{\ell_{f, 0}\ell_{g, 2}}{\mu} + \ell_{g, 1}\left(1+\sqrt{\kappa}\left(\frac{\sqrt{\kappa} - 1}{\sqrt{\kappa} + 1}\right)^N\right)\bigg(\ell_{g, 2}\ell_{f, 0} + \frac{\ell_{f, 1}}{\mu}\bigg)\bigg)\dist(y^*(x^k), y^k)\\& + \ell_{g, 1}\sqrt{\kappa}\left(\frac{\sqrt{\kappa} - 1}{\sqrt{\kappa} + 1}\right)^N \|\hat{v}^0(x^{k}, y^{k})-P_{y^*(x^k)\rightarrow y^k} v^*(x^k)\|.
    \end{split}
    \]
    We obtain the desired result by applying Lemma \ref{lemma_inner} to the above inequality.
\end{proof}

\begin{lemma}\label{lemma2}
    Suppose Assumptions \ref{assumption_1}, \ref{assumption_2}, \ref{assumption_3} and \ref{assumption_4} hold, then RieBO satisfies:
    \begin{equation}
        \dist(y^{k, 0}, y^*(x^k))^2 + \|P_{y^*(x^k)\rightarrow y^k}v^*(x^k) - \hat{v}^0(x^{k}, y^{k})\|^2 \leq\left(\frac{1}{2}\right)^{k} \Delta_{0}+\Omega \sum_{j=0}^{k-1}\left(\frac{1}{2}\right)^{k-1-j}\left\|\grad \Phi\left(x^{j}\right)\right\|^{2},
    \end{equation}
    with the following choice of parameters:
    \begin{equation}
        \begin{aligned}
        &T \geq \log\bigg(2 \bigg(7 + 8\kappa^2\alpha^2\Gamma^2\bigg)\bigg(\ell_{g, 2}\ell_{f, 0} + \frac{\ell_{f, 1}}{\mu}\bigg)^2\bigg) / (2\log\bigg( \frac{1}{1-2\mu\tau\beta^2} \bigg)) =\Theta(\kappa), \\
        &N \geq \log\bigg( (4 +16\kappa^2\alpha^2\ell_{g, 1}^2)\kappa \bigg) / (2\log\bigg( \frac{\sqrt{\kappa} - 1}{\sqrt{\kappa} + 1} \bigg)) =\Theta(\sqrt{\kappa}), \\
        &\Omega=\left[2\bigg(\frac{\ell_{f, 0}\ell_{g, 2}}{\mu^2}\sqrt{1+\kappa^2} + \frac{\ell_{f, 1}}{\mu}\bigg)^2+4\kappa^2\right]\alpha^2, \\ &\Delta_{0}=\dist(y^{0, 0}, y^*(x^0))^2 + \|P_{y^*(x^0)\rightarrow y^0}v^*(x^0) - \hat{v}^0(x^{0}, y^{0})\|^2.
        \end{aligned}
    \end{equation}
\end{lemma}
\begin{proof}
    Since $y^{k, 0}=y^{k-1, T}$, we have
    $$
        \dist(y^{k, 0}, y^*(x^{k}))^2\leq  2\dist(y^{k-1, T}, y^*(x^{k-1}))^2 + 2\dist(y^*(x^{k-1}), y^*(x^{k}))^2.
    $$
    Here the first term is again bounded by $(1 - 2\mu\tau\beta^2)^T \dist(y^{k-1, 0}, y^*(x^{k-1}))^2$ by Lemma \ref{lemma_inner}, and the second term is bounded by the Lipschitzness of $y^*$ (Lemma \ref{lemma0}) and by the update in the following way:
    \[
    \begin{split}
        \dist(y^*(x^{k-1}), y^*(x^k))^2\leq \kappa^2 \dist (x^{k-1}, x^{k})^2 = \kappa^2\alpha^2\|h_{\Phi}^{k-1}\|^2.
    \end{split}
    \]
    Thus,
    \[
    \begin{split}
        &\dist(y^{k, 0}, y^*(x^k))^2\leq 2\dist(y^{k-1, T}, y^*(x^{k-1}))^2 + 2\dist(y^*(x^{k-1}), y^*(x^{k}))^2 \\
        \leq & 2(1 - 2\mu\tau\beta^2)^T \dist(y^{k-1, 0}, y^*(x^{k-1}))^2 + 2\kappa^2\alpha^2\|h_{\Phi}^{k-1}\|^2 \\
        \leq & 2(1 - 2\mu\tau\beta^2)^T \dist(y^{k-1, 0}, y^*(x^{k-1}))^2 + 4\kappa^2\alpha^2\|h_{\Phi}^{k-1} - \grad\Phi(x^{k-1})\|^2 + 4\kappa^2\alpha^2\|\grad\Phi(x^{k-1})\|^2 \\
        \leq & \bigg(2 + 8\kappa^2\alpha^2\Gamma^2\bigg) (1-2\mu\tau\beta^2)^{T} \dist(y^*(x^{k-1}), y^{k-1, 0})^2 \\&+ 8\kappa^2\alpha^2\ell_{g, 1}^2\kappa\left(\frac{\sqrt{\kappa} - 1}{\sqrt{\kappa} + 1}\right)^{2 N} \|\hat{v}^0(x^{k-1}, y^{k-1})-\tilde{v}^{k-1}\|^2 + 4\kappa^2\alpha^2\|\grad\Phi(x^{k-1})\|^2 \\
        \leq & \bigg(2 + 8\kappa^2\alpha^2\Gamma^2\bigg) (1-2\mu\tau\beta^2)^{T} \dist(y^*(x^{k-1}), y^{k-1, 0})^2 + 4\kappa^2\alpha^2\|\grad\Phi(x^{k-1})\|^2\\&+ 16\kappa^2\alpha^2\ell_{g, 1}^2\kappa\left(\frac{\sqrt{\kappa} - 1}{\sqrt{\kappa} + 1}\right)^{2 N} \|\hat{v}^0(x^{k-1}, y^{k-1})-P_{y^*(x^{k-1})\rightarrow y^{k-1}} v^*(x^{k-1})\|^2  \\&+ 16\kappa^2\alpha^2\ell_{g, 1}^2\kappa\left(\frac{\sqrt{\kappa} - 1}{\sqrt{\kappa} + 1}\right)^{2 N}  \|P_{y^*(x^{k-1})\rightarrow y^{k-1}}v^*(x^{k-1})-\tilde{v}^{k-1}\|^2,
    \end{split}
    \]
    where the third inequality is by Lemma \ref{lemma1}. For the last term, by \eqref{eq_bound_tilde_v_v_star} we have
    \begin{equation}\label{eq_lemma2_temp0}
        \|P_{y^*(x^{k-1})\rightarrow y^k}v^*(x^{k-1}) - \Tilde{v}^{k-1}\|^2 \leq \bigg(\ell_{g, 2}\ell_{f, 0} + \frac{\ell_{f, 1}}{\mu}\bigg)^2\dist(y^{k-1}, y^*(x^{k-1}))^2.
    \end{equation}
    Thus we have
    \[
    \begin{split}
        &\dist(y^{k, 0}, y^*(x^k))^2 \\
        \leq & \bigg(2 + 8\kappa^2\alpha^2\Gamma^2\bigg) (1-2\mu\tau\beta^2)^{T} \dist(y^*(x^{k-1}), y^{k-1, 0})^2 + 4\kappa^2\alpha^2\|\grad\Phi(x^{k-1})\|^2\\
        &+ 16\kappa^2\alpha^2\ell_{g, 1}^2\kappa\left(\frac{\sqrt{\kappa} - 1}{\sqrt{\kappa} + 1}\right)^{2 N} \|\hat{v}^0(x^{k-1}, y^{k-1})-P_{y^*(x^{k-1})\rightarrow y^{k-1}} v^*(x^{k-1})\|^2 \\&+ 16\kappa^2\alpha^2\ell_{g, 1}^2\kappa\left(\frac{\sqrt{\kappa} - 1}{\sqrt{\kappa} + 1}\right)^{2 N}\bigg(\ell_{g, 2}\ell_{f, 0} + \frac{\ell_{f, 1}}{\mu}\bigg)^2\dist(y^{k-1}, y^*(x^{k-1}))^2 \\
        \leq & \bigg[2 + 8\kappa^2\alpha^2\Gamma^2+16\kappa^2\alpha^2\ell_{g, 1}^2\kappa\left(\frac{\sqrt{\kappa} - 1}{\sqrt{\kappa} + 1}\right)^{2 N}\bigg(\ell_{g, 2}\ell_{f, 0} + \frac{\ell_{f, 1}}{\mu}\bigg)^2\bigg] (1-2\mu\tau\beta^2)^{T} \dist(y^*(x^{k-1}), y^{k-1, 0})^2 \\
        &+ 16\kappa^2\alpha^2\ell_{g, 1}^2\kappa\left(\frac{\sqrt{\kappa} - 1}{\sqrt{\kappa} + 1}\right)^{2 N} \|\hat{v}^0(x^{k-1}, y^{k-1})-P_{y^*(x^{k-1})\rightarrow y^{k-1}} v^*(x^{k-1})\|^2+ 4\kappa^2\alpha^2\|\grad\Phi(x^{k-1})\|^2.
    \end{split}
    \]
    Now we bound $\|P_{y^*(x^k)\rightarrow y^k}v^*(x^k) - \hat{v}^0(x^{k}, y^{k})\|^2$. We have
    \begin{equation}\label{eq_lemma2_temp1}
    \begin{split}
        &\|P_{y^*(x^k)\rightarrow y^k}v^*(x^k) - \hat{v}^0(x^{k}, y^{k})\|^2 =\|P_{y^*(x^k)\rightarrow y^k}v^*(x^k) - P_{y^{k-1}\rightarrow y^k}\hat{v}^N(x^{k-1}, y^{k-1})\|^2 \\
        \leq & 2\|P_{y^*(x^k)\rightarrow y^k}v^*(x^k) - P_{y^*(x^{k-1})\rightarrow y^k}v^*(x^{k-1})\|^2 + 2\|P_{y^*(x^{k-1})\rightarrow y^k}v^*(x^{k-1}) - P_{y^{k-1}\rightarrow y^k}\hat{v}^N(x^{k-1}, y^{k-1})\|^2\\
        \leq & 2\|P_{y^*(x^k)\rightarrow y^*(x^{k-1})}v^*(x^k) - v^*(x^{k-1})\|^2 + 4\|P_{y^*(x^{k-1})\rightarrow y^{k-1}}v^*(x^{k-1}) - \Tilde{v}^{k-1}\|^2 + 4\|\Tilde{v}^{k-1} - \hat{v}^N(x^{k-1}, y^{k-1})\|^2 \\
        \leq &4\kappa\left(\frac{\sqrt{\kappa} - 1}{\sqrt{\kappa} + 1}\right)^{2 N}\|\Tilde{v}^{k-1} - \hat{v}^0(x^{k-1}, y^{k-1})\|^2 \\&+ 2\|P_{y^*(x^k)\rightarrow y^*(x^{k-1})}v^*(x^k) - v^*(x^{k-1})\|^2 + 4\|P_{y^*(x^{k-1})\rightarrow y^k}v^*(x^{k-1}) - \Tilde{v}^{k-1}\|^2 \\
        \leq &4\kappa\left(\frac{\sqrt{\kappa} - 1}{\sqrt{\kappa} + 1}\right)^{2 N}\|\Tilde{v}^{k-1} - P_{y^*(x^{k-1})\rightarrow y^{k-1}}v^*(x^{k-1}) \|^2 \\&+ 4\kappa\left(\frac{\sqrt{\kappa} - 1}{\sqrt{\kappa} + 1}\right)^{2 N}\|P_{y^*(x^{k-1})\rightarrow y^{k-1}}v^*(x^{k-1}) - \hat{v}^0(x^{k-1}, y^{k-1})\|^2 \\&+ 2\|P_{y^*(x^k)\rightarrow y^*(x^{k-1})}v^*(x^k) - v^*(x^{k-1})\|^2 + 4\|P_{y^*(x^{k-1})\rightarrow y^{k-1}}v^*(x^{k-1}) - \Tilde{v}^{k-1}\|^2 \\
        = & 4\kappa\left(\frac{\sqrt{\kappa} - 1}{\sqrt{\kappa} + 1}\right)^{2 N}\|P_{y^*(x^{k-1})\rightarrow y^{k-1}}v^*(x^{k-1}) - \hat{v}^0(x^{k-1}, y^{k-1})\|^2 \\&+ 2\|P_{y^*(x^k)\rightarrow y^*(x^{k-1})}v^*(x^k) - v^*(x^{k-1})\|^2 + 4\left(\kappa(\frac{\sqrt{\kappa} - 1}{\sqrt{\kappa} + 1})^{2 N} + 1\right)\|P_{y^*(x^{k-1})\rightarrow y^{k-1}}v^*(x^{k-1}) - \Tilde{v}^{k-1}\|^2,
    \end{split}
    \end{equation}
    where in the second last inequality we again used Lemma \ref{lemma_cg_error}. Now we inspect the two terms in the last line above. Note that the last term is bounded in \eqref{eq_lemma2_temp0}. For the first term, by \eqref{eq_bound_v_star} we have
    $$
    \|P_{y^*(x^k)\rightarrow y^*(x^{k-1})}v^*(x^k) - v^*(x^{k-1})\|^2\leq \bigg(\frac{\ell_{f, 0}\ell_{g, 2}}{\mu^2}\sqrt{1+\kappa^2} + \frac{\ell_{f, 1}}{\mu}\bigg)^2\alpha^2 \|\grad\Phi(x^{k-1})\|^2.
    $$
    Now plugging everything back to \eqref{eq_lemma2_temp1} we get
    \begin{equation}
    \begin{split}
        &\|P_{y^*(x^k)\rightarrow y^k}v^*(x^k) - \hat{v}^0(x^{k}, y^{k})\|^2\\ \leq & 4\kappa\left(\frac{\sqrt{\kappa} - 1}{\sqrt{\kappa} + 1}\right)^{2 N}\|P_{y^*(x^{k-1})\rightarrow y^{k-1}}v^*(x^{k-1}) - \hat{v}^0(x^{k-1}, y^{k-1})\|^2 \\&+ 2\bigg(\frac{\ell_{f, 0}\ell_{g, 2}}{\mu^2}\sqrt{1+\kappa^2} + \frac{\ell_{f, 1}}{\mu}\bigg)^2\alpha^2 \|\grad\Phi(x^{k-1})\|^2 \\&+ 4\left(\kappa(\frac{\sqrt{\kappa} - 1}{\sqrt{\kappa} + 1})^{2 N} + 1\right)\bigg(\ell_{g, 2}\ell_{f, 0} + \frac{\ell_{f, 1}}{\mu}\bigg)^2(1-2\mu\tau\beta^2)^{T}\dist(y^{k-1, 0}, y^*(x^{k-1}))^2,
    \end{split}
    \end{equation}
    where we also used Lemma \ref{lemma_inner} in the last inequality. Now summing up the bound for $\dist(y^{k, 0}, y^*(x^k))^2$ and $\|P_{y^*(x^k)\rightarrow y^k}v^*(x^k) - \hat{v}^0(x^{k}, y^{k})\|^2$, we get:
    \[
    \begin{split}
        &\dist(y^{k, 0}, y^*(x^k))^2 + \|P_{y^*(x^k)\rightarrow y^k}v^*(x^k) - \hat{v}^0(x^{k}, y^{k})\|^2\\ 
        \leq & C_1(1-2\mu\tau\beta^2)^{T}\dist(y^{k-1, 0}, y^*(x^{k-1}))^2 \\
        &+ C_2\|P_{y^*(x^{k-1})\rightarrow y^{k-1}}v^*(x^{k-1}) - \hat{v}^0(x^{k-1}, y^{k-1})\|^2 \\ 
        &+ \left[2\bigg(\frac{\ell_{f, 0}\ell_{g, 2}}{\mu^2}\sqrt{1+\kappa^2} + \frac{\ell_{f, 1}}{\mu}\bigg)^2+4\kappa^2\right]\alpha^2 \|\grad\Phi(x^{k-1})\|^2,
    \end{split}
    \]
    with
    \[
    \begin{split}
        &C_1 = \bigg(6 + 8\kappa^2\alpha^2\Gamma^2+C_2\bigg)\bigg(\ell_{g, 2}\ell_{f, 0} + \frac{\ell_{f, 1}}{\mu}\bigg)^2\\
        &C_2 = \bigg(4 +16\kappa^2\alpha^2\ell_{g, 1}^2\bigg)\kappa\left(\frac{\sqrt{\kappa} - 1}{\sqrt{\kappa} + 1}\right)^{2 N}.
    \end{split}
    \]
    Now consider the choice of $T$ and $N$ in the statement of this lemma, we can guarantee that $C_1, C_2\leq 1/2$, thus
    \[
    \begin{split}
        &\dist(y^{k, 0}, y^*(x^k))^2 + \|P_{y^*(x^k)\rightarrow y^k}v^*(x^k) - \hat{v}^0(x^{k}, y^{k})\|^2\\ \leq & \frac{1}{2}(\dist(y^{k-1, 0}, y^*(x^{k-1}))^2 + \|P_{y^*(x^k)\rightarrow y^k}v^*(x^k) - \hat{v}^0(x^{k}, y^{k})\|^2) + \Omega \|\grad\Phi(x^{k-1})\|^2.
    \end{split}
    \]
    The final result is obtained by taking the telescoping sum of the above inequality.
\end{proof}

\begin{lemma}\label{lemma3}
    Suppose the parameters are set the same as in Lemma \ref{lemma2}, then we have
    \begin{equation}
        \|h_{\Phi}^k - {\grad}\Phi(x^{k})\|^2 \leq \delta_{T, N}\left(\frac{1}{2}\right)^{k} \Delta_{0}+\delta_{T, N} \Omega \sum_{j=0}^{k-1}\left(\frac{1}{2}\right)^{k-1-j}\left\|\grad \Phi\left(x^{j}\right)\right\|^{2},
    \end{equation}
    where
    \begin{equation}
    \delta_{T, N}=2\Gamma^2 (1-2\mu\tau\beta^2)^{T}\ell_{g, 1}^2\kappa\left(\frac{\sqrt{\kappa} - 1}{\sqrt{\kappa} + 1}\right)^{2 N}.
    \end{equation}
\end{lemma}
\begin{proof}
    By Lemma \ref{lemma1} and $ab+cd\leq(a+c)(b+d)$ for any positive $a,b,c,d$, we have
    $$
        \|h_{\Phi}^k - {\grad}\Phi(x^{k})\|^{2} \leq \delta_{T, N}\left(\dist(y^*(x^k), y^{k-1})^2 + \|\hat{v}^0(x^{k}, y^{k})-P_{y^*(x^k)\rightarrow y^k} v^*(x^k)\|^2\right).
    $$
    The proof is completed by applying Lemma \ref{lemma2}.
\end{proof}

Now we proceed to the proof of Theorem \ref{theorem1}.

\begin{proof}[Proof of Theorem \ref{theorem1}]
    By Lemma \ref{lemma0}, we have
    \[
    \begin{split}
        \Phi(x^{k+1})&\leq \Phi(x^{k}) + \langle\grad \Phi(x^{k}), \Exp^{-1}_{x^{k}}(x^{k+1})\rangle_{x^{k}} + \frac{L_{\Phi}}{2}\dist(x^{k},x^{k+1})^2 \\
        & = \Phi(x^{k}) -\alpha\langle\grad \Phi(x^{k}), h_{\Phi}^k\rangle_{x^{k}} + \frac{L_{\Phi}\alpha^2}{2}\|h_{\Phi}^k\|_{x^{k}}^2 \\
        &\leq \Phi(x^{k}) - (\frac{\alpha}{2} - \alpha^2L_{\Phi})\|\grad\Phi(x^{k})\|_{x^{k}}^{2} + (\frac{\alpha}{2} + \alpha^2L_{\Phi})\|\grad\Phi(x^{k}) - h_{\Phi}^k\|_{x^{k}}^{2}.
    \end{split}
    \]
    Now by using Lemma \ref{lemma3}, we get
    \[
    \begin{split}
        \Phi(x^{k+1}) \leq& \Phi(x^{k}) - (\frac{\alpha}{2} - \alpha^2L_{\Phi})\|\grad\Phi(x^{k})\|_{x^{k}}^{2} \\
        & + (\frac{\alpha}{2} + \alpha^2L_{\Phi})\left[ \delta_{T, N}\left(\frac{1}{2}\right)^{k} \Delta_{0}+\delta_{T, N} \Omega \sum_{j=0}^{k-1}\left(\frac{1}{2}\right)^{k-1-j}\left\|\grad \Phi\left(x^{j}\right)\right\|^{2} \right].
    \end{split}
    \]
    Now by taking the telescoping sum of the above inequality over $k$ from $0$ to $K-1$, we have
    $$
        \begin{aligned}
        \left(\frac{\alpha}{2}-\alpha^{2} L_{\Phi}\right) \sum_{k=0}^{K-1}\left\|\grad \Phi\left(x^{k}\right)\right\|^{2} \leq \Phi\left(x_{0}\right)-\inf _{x\in\M} \Phi(x)+\left(\frac{\alpha}{2}+\alpha^{2} L_{\Phi}\right) \delta_{T, N} \Delta_{0} \\
        +\left(\frac{\alpha}{2}+\alpha^{2} L_{\Phi}\right) \delta_{T, N} \Omega \sum_{k=1}^{K-1} \sum_{j=0}^{k-1}\left(\frac{1}{2}\right)^{k-1-j}\left\|\grad \Phi\left(x^{j}\right)\right\|^{2}.
        \end{aligned}
    $$
    By the fact that
    \[
    \begin{split}
    \sum_{k=1}^{K-1} \sum_{j=0}^{k-1}\left(\frac{1}{2}\right)^{k-1-j}\left\|\grad \Phi\left(x^j\right)\right\|^2 \leq \sum_{k=0}^{K-1} \frac{1}{2^k} \sum_{k=0}^{K-1}\left\|\grad \Phi\left(x^k\right)\right\|^2 \leq 2 \sum_{k=0}^{K-1}\left\|\grad \Phi\left(x^k\right)\right\|^2,
    \end{split}
    \]
    we have
    $$
        \begin{aligned}
        \left(\frac{\alpha}{2}-\alpha^{2} L_{\Phi} - (\alpha+2\alpha^{2} L_{\Phi})\delta_{T, N} \Omega \right) \sum_{k=0}^{K-1}\left\|\grad \Phi\left(x^{k}\right)\right\|^{2} \leq \Phi\left(x_{0}\right)-\inf _{x\in\M} \Phi(x)+\left(\frac{\alpha}{2}+\alpha^{2} L_{\Phi}\right) \delta_{T, N} \Delta_{0}
        \end{aligned}.
    $$
    Choosing $N\geq\Theta(\sqrt{\kappa})$ and $D\geq\Theta(\kappa)$ as in Lemma \ref{lemma2}, we are able to ensure that 
    $$
    \Omega\left(1+2  \alpha L_{\Phi}\right) \delta_{T, N} \leq \frac{1}{4}, \quad \delta_{T, N} \leq 1. 
    $$
    As a result, we get
    $$
    \left(\frac{\alpha}{4}-\alpha^{2} L_{\Phi} \right) \sum_{k=0}^{K-1}\left\|\grad \Phi\left(x^{k}\right)\right\|^{2} \leq \Phi\left(x_{0}\right)-\inf _{x\in\M} \Phi(x) + \left(\frac{\alpha}{2}+\alpha^{2} L_{\Phi}\right) \Delta_{0}.
    $$
    Thus, with $\alpha\leq\frac{1}{8L_{\Phi}}$ we get
    $$
    \frac{1}{K} \sum_{k=0}^{K-1}\left\|\grad \Phi\left(x^{k}\right)\right\|^{2} \leq \frac{64 L_{\Phi}\left(\Phi\left(x_{0}\right)-\inf _{x} \Phi(x)\right)+5 \Delta_{0}}{K}.
    $$
    Now we inspect the oracle complexities. To ensure $\frac{1}{K} \sum_{k=0}^{K-1}\left\|\grad \Phi\left(x^{k}\right)\right\|^{2} \leq\epsilon$, we need $K= \mathcal{O}(\frac{\kappa^3}{\epsilon})$, so that $\operatorname{Gc}(f,\epsilon)=\mathcal{O}(\frac{\kappa^3}{\epsilon})$. Since in each outer iteration, we need $D=\mathcal{O}(\kappa)$ iterations, so $\operatorname{Gc}(g,\epsilon)=\mathcal{O}(\frac{\kappa^4}{\epsilon})$. The Jacobian-vector product count is the same as the iteration number $K$ since it is only conducted once for every iteration. The Hessian-vector product is conducted for $N=\mathcal{O}(\sqrt{\kappa})$ times for each iteration. Thus we have the previously described complexities.
\end{proof}

\section{Stochastic Algorithm RieSBO and Its Convergence}

In this section, we propose RieSBO (Algorithm \ref{algo_stoc_bilevel_AID_ITD}) for stochastic bilevel manifold optimization \eqref{stochastic_problem}. The algorithm is a generalization of its counterpart in the Euclidean space as in \cite{hong2020two,chen2021tighter}, where we employ the Neumann series estimation for the hypergradient as in \eqref{approximate_Neumann_stochastic}.

\begin{algorithm}[!ht]
\SetKwInOut{Input}{input}
\SetKwInOut{Output}{output}
\SetAlgoLined
\Input{$K$, $T$, $Q$, stepsize $\{\alpha_k,\beta_k\}$, initializations $x^0\in\M, y^0\in\N$}
    \For{$k=0,1,2,...,K-1$}{
        Set $y^{k, 0}=y^{k-1}$\;
        \For{$t=0,...,T-1$}{
            Update $y^{k, t+1}\leftarrow \Exp_{y^{k, t}}(-\beta_k \Tilde{h}_{g}^{k, t})$ with $\Tilde{h}_{g}^{k, t}:=\grad_{y}G(x^k,y^{k, t};\zeta_{k, t})$\;
        }
        Set $y^{k}\leftarrow y^{k, T}$\;
        
        Update $x^{k+1}\leftarrow\Exp_{x^{k}}(-\alpha_k \Tilde{h}_{\Phi}^k)$,  where $\Tilde{h}_{\Phi}^k$ is as defined in \eqref{approximate_Neumann_stochastic}\;
    }
 \label{algo_stoc_bilevel_AID_ITD}
 \caption{Algorithm for {\bf Rie}mannian {\bf S}tochastic {\bf B}ilevel {\bf O}ptimization (\bf{RieSBO})}
\end{algorithm}

For the stochastic case, we utilize the following notion of stationarity.
\begin{definition}\label{def_stochastic_eps_stationary}
    A random point $x\in\M$ is called an $\epsilon$-stationary point for \eqref{stochastic_problem} if $\E\|\nabla\Phi(x)\|^2\leq\epsilon$.
\end{definition}

We now proceed to the convergence analysis for the Riemannian stochastic bilevel optimization (RieSBO, Algorithm \ref{algo_stoc_bilevel_AID_ITD}). For RieSBO, we need the following additional assumption over the mean and variance of the estimators.
\begin{assumption}\label{assumption_5}
    The stochastic gradients satisfy $\grad F(x,y;\xi) = [\grad_x F(x,y;\xi), \grad_y F(x,y;\xi)]$ and $\grad G(x,y;\zeta)= [\grad_x G(x,y;\zeta), \grad_y G(x,y;\zeta)]$. The second order gradients $\grad_{x, y}^2 G(x,y;\zeta)$, $H_y(G(x,y;\zeta))$ are all unbiased estimators of the corresponding deterministic quantities of $f$ and $g$. Their variances are all bounded by $\sigma^2$ (in tangent space norms and operator norms, respectively for the Riemannian gradient and Riemannian Hessian).
\end{assumption}
Note that we do not need to assume the smoothness or strong-convexity of the stochastic functions $F$ and $G$. 

Now we are ready to state the following convergence result.
\begin{theorem}\label{theorem3}
    Suppose Assumptions \ref{assumption_1}, \ref{assumption_2}, \ref{assumption_3}, \ref{assumption_4} and \ref{assumption_5} hold. If we take the stepsizes $\alpha_k=\alpha=\frac{1}{\kappa^{5/2}\sqrt{K}}$, $\beta_k=\beta=\min\{\frac{1}{\kappa^{7/4}\sqrt{K}}, \frac{1}{\ell_{g, 1}}\}$, also $\eta = 1/\ell_{g, 1}, Q=\mathcal{O}(\kappa\log K)$ and $T=\mathcal{O}(\kappa^4)$. Also suppose that the random variables for all iterations $\zeta_k^t$, $\zeta_{k,(q)}$, $\xi_k$ are i.i.d. samples, then RieSBO (Algorithm \ref{algo_stoc_bilevel_AID_ITD}) satisfies
    $$
        \frac{1}{K}\sum_{k=0}^{K-1}\E[\|\grad\Phi(x^{k})\|^2]\leq \mathcal{O}\left(\frac{\kappa^{2.5}}{\sqrt{K}}\right).
    $$
    Here the expectation is taken with respect to all the random samples. In order to obtain an $\epsilon$-stationary point, i.e., $\frac{1}{K}\sum_{k=0}^{K-1}\E[\|\grad\Phi(x^{k})\|^2]\leq\epsilon$, the oracle complexities needed are given by:
    \begin{itemize}
        \item Gradients: $\operatorname{Gc}(f,\epsilon)=\mathcal{O}(\kappa^5\epsilon^{-2})$, $\operatorname{Gc}(g,\epsilon)=\mathcal{O}(\kappa^9\epsilon^{-2})$;
        
        \item Jacobian and Hessian-vector products: $\operatorname{JV}(g, \epsilon)=\mathcal{O}(\kappa^5\epsilon^{-2})$, $\operatorname{HV}(g, \epsilon)=\mathcal{O}(\kappa^{6}\epsilon^{-2})$.
    \end{itemize}
\end{theorem}

To prove this theorem, we need the following lemmas. For simplicity, denote $\mathcal{U}_{k}$ the $\sigma$-algebra generated by all the random samples up to the $(k-1)$-th iterate, and denote $\Bar{h}_{\Phi}^k:=\E[\Tilde{h}_{\Phi}^k\mid\mathcal{U}_{k}]$, i.e., the expectation only with respect to the samples of the current iterate. 

\begin{lemma}\label{lemma6}
    Suppose we estimate the hypergradient $\Tilde{h}_{\Phi}^k$ via (\ref{approximate_Neumann_stochastic}) with $\eta\leq\frac{1}{\ell_{g, 1}}$, then we have the following bounds.
    \begin{equation}\label{lemma6_eq1}
        \E[\|\Tilde{h}_{\Phi}^k - \Bar{h}_{\Phi}^k\|^2\mid\mathcal{U}_{k}]\leq \Tilde{\sigma}^2,
    \end{equation}
    and
    \begin{equation}\label{lemma6_eq2}
        \|\grad\hat{\Phi}(x^{k}) - \Bar{h}_{\Phi}^k\|^2 \leq b_k^2,
    \end{equation}
    where 
    \begin{equation}
    \begin{split}
        &\Tilde{\sigma}^2 := 2\sigma^2 + 6\left(\sigma^2(\sigma^2 + \ell_{f, 0}^2)  +  \ell_{g, 1}^2 (\sigma^2 + \ell_{f, 0}^2)  + \ell_{g, 1}^2\sigma^2\right)\max\{\frac{1}{\mu^2}, \frac{d_1^2}{\eta^2\mu^2}\}=\mathcal{O}(\kappa^2),\\
        &b_k:=\ell_{f, 0}\frac{\ell_{g, 1}}{\mu} (1-\frac{\mu}{\ell_{g, 1}})^{Q},
    \end{split}
    \end{equation}
    and $\hat{\Phi}(x)=f(x, y^T(x))$ which is the approximate function after $T$ steps of the inner loop. 
    
    Further, we have the following bound on the second moment:
    \begin{equation}
        \E[\|\Tilde{h}_{\Phi}^k\|^2\mid\mathcal{U}_{k}]\leq 2\Tilde{\sigma}^2 + 4b_k^2 + 4\ell_{f, 0}^2(1+\kappa)^2 =: \Tilde{C}^2=\mathcal{O}(\kappa^2).
    \end{equation}
\end{lemma}
\begin{proof}[Proof of Lemma \ref{lemma6}]
    By the expression \eqref{grad_Phi} for $\grad\Phi(x)$ and (\ref{stoch_grad_estimate}) for $\Tilde{h}_{\Phi}^k$, we have: 
    \[
    \begin{split}
        &\grad\Phi(x^k) = \grad_x f(x^k,y^*(x^k)) - \grad_{y, x}^2 g(x^k,y^*(x^k))[v^*(x^k)],\\
        &\grad\hat{\Phi}(x^{k}) = \grad_x f(x^k,y^k) - \grad_{y, x}^2 g(x^k,y^k)[\tilde{v}^{k}],\\
        &\Tilde{h}_{\Phi}^k = \grad_x F(x^k,y^k;\xi_k) - \grad_{y, x}^2 G(x^k,y^k;\zeta_{k, (0)})[v_{Q}^k],
    \end{split}
    \]
    where again $\tilde{v}^{k}:=(H_y(g(x^{k},y^{k, T})))^{-1}\grad_y f(x^{k},y^{k, T})$.
    
    For \eqref{lemma6_eq1}, denote
    $$
    \bar{v}_{Q}^k=\E[v_{Q}^k]=\eta \sum_{q=1}^{Q}(I - \eta H_y( g(x^k, y^k)) )^q [\grad_{y} f (x^k, y^k)].
    $$
    We have 
    \begin{equation}\label{lemma6_temp0}
    \begin{split}
        &\E[\|\Tilde{h}_{\Phi}^k - \Bar{h}_{\Phi}^k\|^2\mid\mathcal{U}_{k}] \\\leq & 2\E[\|\grad_{x} f\big(x^{k}, y^{k, T}\big) - \grad_{x} F\big(x^{k}, y^{k, T} ; \xi_k\big)\|^2\mid\mathcal{U}_{k}] \\
        &+2\E[\|\grad_{y, x}^2 G(x^k,y^k;\zeta_{k, (0)})[v_{Q}^k] - \grad_{y, x}^2 g(x^k,y^k)[\bar{v}_{Q}^k]  \|^2\mid\mathcal{U}_{k}] \\
        \leq & 2\sigma^2 +2\E[\|\grad_{y, x}^2 G(x^k,y^k;\zeta_{k, (0)})[v_{Q}^k] - \grad_{y, x}^2 g(x^k,y^k)[\bar{v}_{Q}^k]  \|^2\mid\mathcal{U}_{k}].
    \end{split}
    \end{equation}
    We now inspect the last term above. Denote 
    $$
        H^k:= \eta Q\prod_{q=1}^{Q'}(I - \eta H_y( G(x^k, y^k;\zeta_{k, (q)})) ),
    $$
    which is our estimation of the Riemannian Hessian at the $k$-th outer iteration, and we have that
    \[
    \begin{split}
        &\grad_{y, x}^2 G(x^k,y^k;\zeta_{k, (0)})[v_{Q}^k] - \grad_{y, x}^2 g(x^k,y^k)[\bar{v}_{Q}^k]\\
        = & \grad_{y, x}^2 G(x^k,y^k;\zeta_{k, (0)})\left[H^k[\grad_{y} F (x^k, y^k ; \xi_k)]\right] - \grad_{y, x}^2 g(x^k,y^k)\left[\E[H^k[\grad_{y} F (x^k, y^k ; \xi_k)]]\right] \\
        = & \left\{\grad_{y, x}^2 G(x^k,y^k;\zeta_{k, (0)}) - \grad_{y, x}^2 g(x^k,y^k)\right\}\left[H^k[\grad_{y} F (x^k, y^k ; \xi_k)]\right] \\
        &+ \grad_{y, x}^2 g(x^k,y^k)\left[\left\{H^k - \E[H^k]\right\}[\grad_{y} F (x^k, y^k ; \xi_k)]\right] \\
        &+ \grad_{y, x}^2 g(x^k,y^k)\E[H^k] \left\{ \grad_{y} F (x^k, y^k ; \xi_k) - \grad_{y} f (x^k, y^k) \right\}.
    \end{split}
    \]
    Since
    \begin{equation*}
    \begin{split}
        &\E[\|\grad_{y} F (x^k, y^k ; \xi_k)\|^2] \\=& \E[\|\grad_{y} F (x^k, y^k ; \xi_k) - \grad_{y} f (x^k, y^k)\|^2] + \E[\|\grad_{y} f (x^k, y^k)\|^2] \leq \sigma^2 + \ell_{f, 0}^2,
    \end{split}
    \end{equation*}
    we have that
    \[
    \begin{split}
        &\E[\|\grad_{y, x}^2 G(x^k,y^k;\zeta_{k, (0)})[v_{Q}^k] - \grad_{y, x}^2 g(x^k,y^k)[\bar{v}_{Q}^k]  \|^2\mid\mathcal{U}_{k}]\\
        \leq & 3\sigma^2(\sigma^2 + \ell_{f, 0}^2) \E\|H^k\|_{\mathrm{op}}^2 + 3 \ell_{g, 1}^2 (\sigma^2 + \ell_{f, 0}^2) \E\|H^k-\E[H^k]\|_{\mathrm{op}}^2 + 3\ell_{g, 1}^2\sigma^2\|\E[H^k]\|_{\mathrm{op}}^2.
    \end{split}
    \]
    It remains to bound $\E\|H^k\|_{\mathrm{op}}^2$ and $\|\E[H^k]\|_{\mathrm{op}}$. For $\E\|H^k\|_{\mathrm{op}}^2$, using \citet[Lemma 12]{hong2020two}, we have that
    \begin{equation*}
        \E\|H^k\|_{\mathrm{op}}^2\leq \frac{d_1}{\eta\mu},
    \end{equation*}
    where $d_1> 0$ is some absolute constant. On the other hand $\|\E[H^k]\|_{\mathrm{op}}$ can be easily calculated as (since $\mu\eta< \mu/\ell_{g,1}<1$)
    \[
    \begin{split}
        \|\E[H^k]\|_{\mathrm{op}} =& \eta\|\sum_{q=1}^{Q}(I - \eta H_y( g(x^k, y^k)))^q\|_{\mathrm{op}}\\
        \leq & \|H^{-1}\|_{\mathrm{op}} \|I - \eta H_y( g(x^k, y^k))\|_{\mathrm{op}} \leq \frac{1}{\mu}.
    \end{split}
    \]
    Therefore, we finally have
    \[
    \begin{split}
        &\E[\|\grad_{y, x}^2 G(x^k,y^k;\zeta_{k, (0)})[v_{Q}^k] - \grad_{y, x}^2 g(x^k,y^k)[\bar{v}_{Q}^k]  \|^2\mid\mathcal{U}_{k}]\\
        \leq & 3\left(\sigma^2(\sigma^2 + \ell_{f, 0}^2)  +  \ell_{g, 1}^2 (\sigma^2 + \ell_{f, 0}^2)  + \ell_{g, 1}^2\sigma^2\right)\max\{\frac{1}{\mu^2}, \frac{d_1^2}{\eta^2\mu^2}\}.
    \end{split}
    \]
    Plugging the above equation to \eqref{lemma6_temp0} we get \eqref{lemma6_eq1}.
    
    Now for \eqref{lemma6_eq2}, since
    \[
    \begin{split}
        \Bar{h}_{\Phi}^k:= & \E\left[\grad_x F(x^k,y^k;\xi_k) - \grad_{y, x}^2 G(x^k,y^k;\zeta_{k, (0)})[v_{Q}^k]\right]\\
        &= \grad_x f(x^k,y^k) - \grad_{y, x}^2 g(x^k,y^k)[\bar{v}_{Q}^k],
    \end{split}
    \]
    we have
    \[
    \begin{split}
        &\|\grad\hat{\Phi}(x^{k}) - \Bar{h}_{\Phi}^k\|^2 \leq \|\grad_{y, x}^2 g(x^k,y^k)\|_{\mathrm{op}}^2\|\tilde{v}^{k} - \bar{v}_{Q}^k\|^2 \leq \ell_{g, 1}^2\|\tilde{v}^{k} - \bar{v}_{Q}^k\|^2\\
        = & \ell_{g, 1}^2\|(H_y(g(x^{k},y^{k})))^{-1}[\grad_y f(x^{k},y^{k})] - \eta \sum_{q=1}^{Q}(I - \eta H_y( g(x^k, y^k)) )^q [\grad_{y} f (x^k, y^k)]\|^2\\
        \leq & \ell_{g, 1}^2\ell_{f, 0}^2\|(H_y(g(x^{k},y^{k})))^{-1} - \eta \sum_{q=1}^{Q}(I - \eta H_y( g(x^k, y^k)) )^q\|_{\mathrm{op}}^2\\
        \leq & \ell_{f, 0}^2\frac{\ell_{g, 1}^2}{\mu^2} (1-\frac{\mu}{\ell_{g, 1}})^{2Q}=b_k^2,
    \end{split}
    \]
    where the last line is by \citet[Lemma 3.2]{ghadimi2018approximation}. Note that we take $\eta\leq\frac{1}{\ell_{g, 1}}$ so that the Neumann sequence converges.
    
    Now for the moment $\E[\|\Tilde{h}_{\Phi}^k\|^2\mid\mathcal{U}_{k}]$, we have
    \[
    \begin{split}
        \E[\|\Tilde{h}_{\Phi}^k\|^2\mid\mathcal{U}_{k}]\leq &2\E[\|\Tilde{h}_{\Phi}^k - \Bar{h}_{\Phi}^k\|^2\mid\mathcal{U}_{k}] + 4\|\Bar{h}_{\Phi}^k - \grad\hat{\Phi}(x^{k})\|^2 + 4 \|\grad\hat{\Phi}(x^{k})\|^2\\
        \leq & 2\Tilde{\sigma}^2 + 4b_k^2 + 4\|\grad\hat{\Phi}(x^{k})\|^2.
    \end{split}
    \]
    Since
    \[
    \begin{split}
        \|\grad\hat{\Phi}(x^{k})\| =& \| \grad_x f(x^k,y^k) - \grad_{y, x}^2 g(x^k,y^k)[\tilde{v}^{k}] \| \\
        \leq & \| \grad_x f(x^k,y^k)\|+\|\grad_{y, x}^2 g(x^k,y^k)\|_{\mathrm{op}}\|\tilde{v}^{k}\|\leq \ell_{f, 0} + \ell_{g, 1}\frac{\ell_{f, 0}}{\mu}=\ell_{f, 0}(1+\kappa),
    \end{split}
    \]
    we have
    \[
    \begin{split}
        \E[\|\Tilde{h}_{\Phi}^k\|^2\mid\mathcal{U}_{k}]\leq 2\Tilde{\sigma}^2 + 4b_k^2 + 4\ell_{f, 0}^2(1+\kappa)^2.
    \end{split}
    \]
    This completes the proof. 
\end{proof}

\begin{lemma}\label{lemma5}
    Suppose we have the sequence $\{y^{k, t}\}$ by RieSBO with stepsize $\beta_k=\beta\leq\frac{1}{\ell_{g, 1}}$, then the following inequalities hold:
    \begin{equation}\label{lemma5_eq1}
        \E \dist(y^{k, T}, y^*(x^k))^2\leq (1 - 2\mu\tau\beta^2)^{T}\dist(y^{k, 0}, y^*(x^k))^2 + \tau\beta^2\sigma^2 T,
    \end{equation}
    and
    \begin{equation}\label{lemma5_eq2}
    \begin{aligned}
        &\E[\dist(y^{k, T}, y^*(x^{k+1}))^2]\\
        \leq & 2(1 - 2\mu\tau\beta^2)^{T}\dist(y^{k, 0}, y^*(x^k))^2 + 2\tau\beta^2\sigma^2 T + 4\tau \kappa^2\alpha^2\|\Bar{h}_{\Phi}^k\|_{x^{k}}^2 + 4\tau \kappa^2\alpha^2\Tilde{\sigma}^2.
    \end{aligned}
    \end{equation}
\end{lemma}
\begin{proof}[Proof of Lemma \ref{lemma5}]
    For simplicity, all the expectations are conditioned on $\mathcal{U}_{k}$ in this proof.
    
    First we have by \citet[Corollary 8]{zhang2016first} that
    \[
    \begin{split}
        & \E_{\zeta_{k, t}}\dist(y^{k, t+1}, y^*(x^k))^2\\ 
        \leq & \dist(y^{k, t}, y^*(x^k))^2 + 2\beta\langle \grad g(x^k, y^k), \Exp_{y^{k, t}}(y^*(x^k))\rangle + \tau\beta^2\E_{\zeta_{k, t}}\|\Tilde{h}_{g}^{k, t}\|^2 \\
        \leq & \dist(y^{k, t}, y^*(x^k))^2 + 2\beta\langle \grad g(x^k, y^k), \Exp_{y^{k, t}}(y^*(x^k))\rangle + \tau\beta^2\|\grad g(x^k, y^k)\|^2+ \tau\beta^2\sigma^2 \\
        \leq & (1 - 2\mu\tau\beta^2)\dist(y^{k, t}, y^*(x^k))^2 + \tau\beta^2\sigma^2,
    \end{split}
    \]
    where in the last line we used the same trick as the proof of Lemma \ref{lemma_inner}. Note that in the above formulas the expectation is only taken with respect to the random variables in $\Tilde{h}_{g}^{k, t}$, i.e., $\zeta_{k, t}$.
    Repeating this for $T$ times yields \eqref{lemma5_eq1}.
    Now for the second inequality \eqref{lemma5_eq2}, we have
    \begin{equation}\label{lemma5_temp1}
    \begin{aligned}
        &\E[\dist(y^{k, T}, y^*(x^{k+1}))^2]\\
        \leq& 2\E[\dist(y^{k, T}, y^*(x^{k}))^2] + 2\E\dist(y^*(x^{k}), y^*(x^{k+1}))^2 \\
        \leq & 2(1 - 2\mu\tau\beta^2)^{T}\dist(y^{k, 0}, y^*(x^k))^2 + 2\tau\beta^2\sigma^2 T + 2\tau \kappa^2 \E\dist(x^k, x^{k+1})^2,
    \end{aligned}
    \end{equation}
    where the last inequality is by \eqref{lemma5_eq1} and Lemma \ref{lemma0}. For $\E d(x^{k+1}, x^{k})^2$ we have the bound: 
    \[
    \begin{split}
        &\E d(x^{k+1}, x^{k})^2 = \alpha^2\E\|\Tilde{h}_{\Phi}^k\|_{x^{k}}^2 \\
        =&\alpha^2\E\|\Tilde{h}_{\Phi}^k - \Bar{h}_{\Phi}^k + \Bar{h}_{\Phi}^k\|_{x^{k}}^2\leq 2\alpha^2(\|\Bar{h}_{\Phi}^k\|_{x^{k}}^2 + \Tilde{\sigma}^2),
    \end{split}
    \]
    which completes the proof.
\end{proof}

Now we turn to the proof of Theorem \ref{theorem3}.

\begin{proof}[Proof of Theorem \ref{theorem3}]
    Denote $V_k:=\Phi(x^{k}) + \kappa\dist(y^{k-1, T}, y^*(x^{k}))^2$. By Lemma \ref{lemma0} and Lemma \ref{lemma6}, we have
    \[
    \begin{split}
        \E&[\Phi(x^{k+1})\mid\mathcal{U}_{k}]\leq \Phi(x^{k}) + \E[\langle\grad \Phi(x^{k}), \Exp^{-1}_{x^{k}}(x^{k+1})\rangle_{x^{k}}\mid\mathcal{U}_{k}] + \frac{L_{\Phi}}{2}\E[\dist(x^{k},x^{k+1})^2\mid\mathcal{U}_{k}] \\
        =& \Phi(x^{k}) -\alpha\E[\langle\grad \Phi(x^{k}), \Tilde{h}_{\Phi}^k \rangle_{x^{k}}\mid\mathcal{U}_{k}] + \frac{L_{\Phi}\alpha^2}{2}\|\Tilde{h}_{\Phi}^k\|_{x^{k}}^2 \\
        =& \Phi(x^{k}) -\frac{\alpha}{2}\E[\|\grad\Phi(x^{k})\|^2\mid\mathcal{U}_{k}] - (\frac{\alpha}{2}-\frac{\alpha^2 L_{\Phi}}{2})\|\Bar{h}_{\Phi}^k\|^2 + \frac{\alpha}{2}\|\grad\Phi(x^{k}) - \Bar{h}_{\Phi}^k\|^2\\
        &+ \frac{\alpha^2 L_{\Phi}}{2}\E[\|\Tilde{h}_{\Phi}^k - \Bar{h}_{\Phi}^k\|^2\mid\mathcal{U}_{k}] \\
        \leq & \Phi(x^{k}) -\frac{\alpha}{2}\E[\|\grad\Phi(x^{k})\|^2\mid\mathcal{U}_{k}] - (\frac{\alpha}{2}-\frac{\alpha^2 L_{\Phi}}{2})\|\Bar{h}_{\Phi}^k\|^2 + \frac{\alpha}{2}\|\grad\Phi(x^{k}) - \Bar{h}_{\Phi}^k\|^2+ \frac{\alpha^2 L_{\Phi}}{2}\Tilde{\sigma}^2.
    \end{split}
    \]
    Now we decompose the bias term $\|\grad\Phi(x^{k}) - \Bar{h}_{\Phi}^k\|$ as:
    \begin{equation}\label{theorem3_temp0}
    \begin{aligned}
        \|\grad\Phi(x^{k}) - \Bar{h}_{\Phi}^k\|^2 =& 2\|\grad\Phi(x^{k}) - \grad\hat{\Phi}(x^{k})\|^2 + 2\|\grad\hat{\Phi}(x^{k}) - \Bar{h}_{\Phi}^k\|^2\\
        \leq& 2\Gamma_0^2\dist(y^{k, T}, y^*(x^{k}))^2 + 2b_k^2,
    \end{aligned}
    \end{equation}
    where we use a similar process as the proof of Lemma \ref{lemma2} to bound $\|\grad\Phi(x^{k}) - \grad\hat{\Phi}(x^{k})\|$ and $\Gamma_0=\ell_{f, 1} + \frac{\ell_{f, 0}\ell_{g, 2}}{\mu} + \ell_{g, 1}(\ell_{g, 2}\ell_{f, 0} + \frac{\ell_{f, 1}}{\mu})=\mathcal{O}(\kappa)$. Thus we have
    \begin{equation}\label{theorem3_temp1}
    \begin{aligned}
        \E[\Phi(x^{k+1})\mid\mathcal{U}_{k}]\leq&\Phi(x^{k}) -\frac{\alpha}{2}\E[\|\grad\Phi(x^{k})\|^2\mid\mathcal{U}_{k}] - (\frac{\alpha}{2}-\frac{\alpha^2 L_{\Phi}}{2})\|\Bar{h}_{\Phi}^k\|^2\\
        & + \alpha\Gamma_0^2\dist(y^{k, T}, y^*(x^{k}))^2 + \alpha b_k^2+ \frac{\alpha^2 L_{\Phi}}{2}\Tilde{\sigma}^2.
    \end{aligned}
    \end{equation}
    
    Now we have 
    \[
    \begin{split}
        \E[V_{k+1}] &- \E[V_k] = \E[\Phi(x^{k+1})] - \E[\Phi(x^{k})] + \kappa\E\dist(y^{k, T}, y^*(x^{k+1}))^2 - \kappa\E\dist(y^{k-1, T}, y^*(x^{k}))^2 \\
        \leq & -\frac{\alpha}{2}\E[\|\grad\Phi(x^{k})\|^2\mid\mathcal{U}_{k}] - (\frac{\alpha}{2}-\frac{\alpha^2 L_{\Phi}}{2})\E\|\Bar{h}_{\Phi}^k\|^2+ \alpha b_k^2+ \frac{\alpha^2 L_{\Phi}}{2}\Tilde{\sigma}^2\\
        & + \kappa\E\dist(y^{k, T}, y^*(x^{k+1}))^2 - \kappa\E\dist(y^{k-1, T}, y^*(x^{k}))^2 + \alpha\Gamma_0^2\E\dist(y^{k, T}, y^*(x^{k}))^2 \\
        \leq & -\frac{\alpha}{2}\E[\|\grad\Phi(x^{k})\|^2\mid\mathcal{U}_{k}] - (\frac{\alpha}{2}-\frac{\alpha^2 L_{\Phi}}{2})\E\|\Bar{h}_{\Phi}^k\|^2+ \alpha b_k^2+ \frac{\alpha^2 L_{\Phi}}{2}\Tilde{\sigma}^2\\
        & + \kappa\E\bigg(2(1 - 2\mu\tau\beta^2)^{T}\dist(y^{k, 0}, y^*(x^k))^2 + 2\tau\beta^2\sigma^2 T + 4\tau \kappa^2\alpha^2\|\Bar{h}_{\Phi}^k\|_{x^{k}}^2 + 4\tau \kappa^2\alpha^2\Tilde{\sigma}^2\bigg) \\
        &- \kappa\E\dist(y^{k-1, T}, y^*(x^{k}))^2 + \alpha\Gamma_0^2\bigg((1 - 2\mu\tau\beta^2)^{T}\E\dist(y^{k, 0}, y^*(x^k))^2 + \tau\beta^2\sigma^2 T\bigg) \\
        = & -\frac{\alpha}{2}\E[\|\grad\Phi(x^{k})\|^2\mid\mathcal{U}_{k}] - \bigg(\frac{\alpha}{2}-\frac{\alpha^2 L_{\Phi}}{2}-4\tau \kappa^3\alpha^2\bigg)\E\|\Bar{h}_{\Phi}^k\|^2 \\
        &+ \bigg((2\kappa+\alpha\Gamma_0^2)(1 - 2\mu\tau\beta^2)^{T} - \kappa\bigg)\E\dist(y^{k-1, T}, y^*(x^{k}))^2 \\
        &+ \bigg(2\kappa +\alpha\Gamma_0^2\bigg)\tau\beta^2\sigma^2 T + \alpha b_k^2+ (\frac{ L_{\Phi}}{2}+4\tau \kappa^2)\alpha^2\Tilde{\sigma}^2,
    \end{split}
    \]
    where the first inequality is by (\ref{theorem3_temp1}) and the second inequality is by Lemma \ref{lemma5}, as well as the fact that $ y^{k+1, 0}=y^{k, T}$. To make the coefficients negative, notice that by taking
    \begin{equation*}
    \begin{split}
        &\alpha \leq \frac{1}{L_{\Phi} + 8\tau\kappa^2}\\
        &T\geq \log\bigg(\frac{1}{1 - 2\mu\tau\beta^2}\bigg) / \log\bigg(\frac{2\kappa+\alpha\Gamma_0^2}{\kappa}\bigg),
    \end{split}
    \end{equation*}
    we can guarantee
    \begin{equation*}
    \begin{split}
        &\frac{\alpha}{2}-\frac{\alpha^2 L_{\Phi}}{2}-4\tau \kappa^3\alpha^2 \geq 0\\
        &(2\kappa+\alpha\Gamma_0^2)(1 - 2\mu\tau\beta^2)^{T} - \kappa \leq 0.
    \end{split}
    \end{equation*}
    Therefore, we have
    \begin{equation}
    \begin{aligned}
        \E[V_{k+1}] - \E[V_k]\leq & -\frac{\alpha}{2}\E[\|\grad\Phi(x^{k})\|^2\mid\mathcal{U}_{k}] \\
        &+ \bigg(2\kappa +\alpha\Gamma_0^2\bigg)\tau\beta^2\sigma^2 T + \alpha b_k^2+ (\frac{ L_{\Phi}}{2}+4\tau \kappa^2)\alpha^2\Tilde{\sigma}^2.
    \end{aligned}
    \end{equation}
    Note that here we do not need an increasing $T$.
    
    Now taking the telescoping sum of the above inequality for $k=0,...,K-1$, we get
    $$
    \frac{1}{K}\sum_{k=0}^{K-1}\E[\|\grad\Phi(x^{k})\|^2]\leq \frac{2V_0}{\alpha K} + \frac{2}{K}\sum_{k=0}^{K-1}b_k^2 + \bigg(\frac{4\kappa}{\alpha} +2\Gamma_0^2\bigg)\tau\beta^2\sigma^2 T + (L_{\Phi}+8\tau \kappa^2)\alpha\Tilde{\sigma}^2.
    $$
    
    Now since $b_k^2=\ell_{f, 0}^2\kappa^2 (1-\frac{1}{\kappa})^{2 Q}$, the term $ \frac{2}{K}\sum_{k=0}^{K-1}b_k^2=\mathcal{O}(\frac{1}{\sqrt{K}})$ if $Q=\mathcal{O}(\kappa\log(K))$, following the inequality $(1-x)^n\leq e^{-n x}$. If we also select $\alpha_k=\alpha=\frac{1}{\kappa^{5/2}\sqrt{K}}$, $\beta_k=\beta=\min\{\frac{1}{\kappa^{7/4}\sqrt{K}}, \frac{1}{\ell_{g, 1}}\}$ and $T=\mathcal{O}(\kappa^4)$, we are able to get:
    $$
    \frac{1}{K}\sum_{k=0}^{K-1}\E[\|\grad\Phi(x^{k})\|^2]\leq \mathcal{O}(\frac{\kappa^{2.5}}{\sqrt{K}}).
    $$
    
    Now we inspect the oracle complexities. To ensure $\frac{1}{K}\sum_{k=0}^{K-1}\E[\|\grad\Phi(x^{k})\|^2]\leq\epsilon$, we need $K=\mathcal{O}(\kappa^5\epsilon^{-2})$, thus $\operatorname{Gc}(F,\epsilon)=\mathcal{O}(\kappa^5\epsilon^{-2})$; Also $\operatorname{Gc}(G,\epsilon)=K T=\mathcal{O}(\kappa^9\epsilon^{-2})$.
\end{proof}

\begin{remark}\label{rmk_hessian_estimator}
    Note that the trick for estimating the Hessian-vector product \eqref{vq} can also be applied to the deterministic case, without using conjugate gradient method, leading to an easier implementation. We just need to replace the stochastic functions in \eqref{vq} by their deterministic versions. In the experiments, we always use \eqref{vq} in this way instead of solving \eqref{AID_subproblem_approx} which uses $N$-step conjugate gradient method, while still achieving reasonable results numerically.
\end{remark}

\section{Numerical experiments on robust optimization on manifolds}
Consider the robust optimization on manifolds:
\begin{equation}
    \min_{x\in\M}\max_{y\in \Delta_n} \sum_{i=1}^{n} y_{i} \ell\left(x ; \xi_{i}\right) - \lambda\left\|{y}-\frac{\mathbf{1}}{n}\right\|^{2},
\end{equation}
where $\Delta_n:=\left\{y \in \mathbb{R}^{n}: \sum_{i=1}^{n} y_{i}=1, y_{i} \geq 0\right\}$ is the probability simplex, and $\ell$ is geodesically convex. This problem minimizes $n$ loss function by dynamically assigning different weights to them, and making sure that the larger loss has larger weights (see \cite{chen2017robust,huang2020gradient}). By minimax theorem we can exchange the min and max of the problem, thus it can be equivalently formulated as a bilevel optimization as follows:
\begin{equation}\label{robust_manifold_problem}
\begin{aligned}
    & \min_{y\in\Delta_n}\ \lambda\left\|{y}-\frac{\mathbf{1}}{n}\right\|^{2} - \sum_{i=1}^{n} y_{i} \ell\left(x ; \xi_{i}\right) \\
    & \text{ s.t. } x\in \argmin_{x\in\M}\ \sum_{i=1}^{n} y_{i} \ell\left(x ; \xi_{i}\right).
\end{aligned}
\end{equation}
{It is worth noticing that having an constraint set in the upper level problem is not covered in our theoretical analysis due to the fact that existing constrained Riemannian optimization techniques such as (stochastic) Riemannian Frank-Wolfe (see \cite{weber2022projection,weber2023riemannian}) require a mini-batch sampling technique, i.e., using \eqref{approximate_Neumann_stochastic} multiple times and taking the average of these estimators to estimate $\grad \Phi(x^k)$ to reduce the variance, which is not desirable in practice. Instead, we point out that since the upper level in \eqref{robust_manifold_problem} is a constrained optimization in a Euclidean space, one could utilize the analysis in \cite{hong2020two} to achieve a similar convergence result as the unconstrained case in \eqref{deterministic_problem}. Therefore we simply add a projection step for the upper level update, and we still observed reasonable convergence results. We present the algorithm we use for the numerical experiments in Algorithm \ref{algo_bilevel_robust}. Note that here the variables in the upper and lower level problems are respectively denoted by $y$ and $x$, which is different from the previous algorithms. It is also worth noticing that the convergence criteria are altered due to the existence of the projection step: in Algorithm \ref{algo_bilevel_robust}, we simply measure the norm of the quantity
$$
\mathcal{G}^k:=\frac{1}{\alpha_k}(y^{k} - y^{k+1})=\frac{1}{\alpha_k}(y^{k} - \proj_{\Delta_n}(y^{k}-\alpha_k h_{\Phi}^k)),
$$
which we refer to as the approximate gradient mapping. This quantity can be used for approximately measuring the stationarity since if we do not have a constraint,
$$
\mathcal{G}^k=\frac{1}{\alpha_k}(y^{k} - y^{k+1})=h_{\Phi}^k\approx \grad \Phi(y^k),
$$
based on Lemma \ref{lemma6}.
}

\begin{algorithm}[!ht]
\SetKwInOut{Input}{input}
\SetKwInOut{Output}{output}
\SetAlgoLined
\Input{$K$, $T$, $N$(steps for conjugate gradient), stepsize $\{\alpha_k,\beta_k\}$, initializations $y^0\in\Delta_n, x^0\in\N$}
    \For{$k=0,1,2,...,K-1$}{
        Set $x^{k, 0}=x^{k-1}$\;
        \For{$t=0,...,T-1$}{
            Update $x^{k, t+1}\leftarrow \Exp_{x^{k, t}}(-\beta_k h_{g}^{k, t})$ with $h_{g}^{k, t}:=\grad_y g(y^k, x^{k,t})$ \;
        }
        Set $x^{k}\leftarrow x^{k, T}$\;
        
        Update $y^{k+1}\leftarrow\proj_{\Delta_n}(y^{k}-\alpha_k h_{\Phi}^k)$ as in \eqref{vq}, in the view of Remark \ref{rmk_hessian_estimator}\;
    }
 \label{algo_bilevel_robust}
 \caption{Bilevel algorithm for robust manifold optimization problem \eqref{robust_manifold_problem}}
\end{algorithm}

We test our proposed algorithms on two concrete examples that lie in this scope: the robust Karcher mean problem and the robust covariance matrix estimation problems. In both experiments we only consider the deterministic function to test the efficacy of the proposed algorithm framework. In both experiments, we use RieBO (Algorithm \ref{algo_bilevel_AID_ITD}) while utilizing the trick in Remark \ref{rmk_hessian_estimator} to estimate the Hessian-vector products.

\subsection{Robust Karcher mean problem}
For the robust Karcher mean problem, one seeks to solve
\begin{equation}\label{robust_karcher}
\begin{aligned}
    & \min_{y\in\Delta_n}\ \left\|y-\frac{\mathbf{1}}{n}\right\|^{2} - \sum_{i=1}^{n} y_{i} \dist(S, A_i)^2\\
    & \text{ s.t. } S\in \argmin_{S\in\mathbb{S}^{d}_{++}}\ \sum_{i=1}^{n} y_{i} \dist(S, A_i)^2,
\end{aligned}
\end{equation}
where $A_i$'s are the symmetric positive definite data matrices, and $\dist(A, B):=\|\log(A^{-1/2}B A^{-1/2})\|_{F}$ is the geodesic distance of two positive definite matrices (see \citet[Chapter 6]{bhatia2009positive}). The squared geodesic distance guarantees the geodesic strong convexity of the lower level problem (see \cite{zhang2016first}), which further ensures that the bilevel problem \eqref{robust_karcher} is well-defined. For the function $h(S):=\dist(S, A)^2$, we have the Euclidean and Riemannian gradients as (see \cite{ferreira2019gradient}, \citet[Chapter 6]{bhatia2009positive}):
\begin{equation*}
\begin{aligned}
    & \nabla_S h(S) = S^{-1/2}\log(S^{1/2}A^{-1} S^{1/2}) S^{-1/2}, \\
    & \grad_S h(S) = S \nabla_S h(S) S = S^{1/2}\log(S^{1/2}A^{-1} S^{1/2}) S^{1/2}.
\end{aligned}
\end{equation*}
The Euclidean and Riemannian Hessian of $h(S):=\dist(S, A)^2$ are less straightforward to calculate, and to the best of our knowledge, they do not exist in the literature. Here we propose an implementable way to calculate it: first notice that (see \citet[Chapter 6]{bhatia2009positive})
$$
\nabla_S h(S) = S^{-1/2}\log(S^{1/2}A^{-1} S^{1/2}) S^{-1/2} = S^{-1} A^{1/2} \log(A^{-1/2} S A^{-1/2})A^{-1/2}.
$$
For any symmetric matrix $V$, we have
$$
\langle\nabla_S h(S), V\rangle = \tr (V S^{-1} A^{1/2} \log(A^{-1/2} S A^{-1/2})A^{-1/2}).
$$
To take the derivative of this, notice that $S$ appears twice in $\langle\nabla_S h(S), V\rangle$. Denote $\Tilde{h}(S_1, S_2) = \tr (V S_1^{-1} A^{1/2} \log(A^{-1/2} S_2 A^{-1/2})A^{-1/2})$, we know that (see \cite{petersen2008matrix})
$$
\frac{\partial \Tilde{h}}{\partial S_1} = -S_1^{-1} V A^{-1/2} \log(A^{-1/2} S_2 A^{-1/2})A^{1/2} S_1^{-1}.
$$

It remains to calculate ${\partial \Tilde{h}} / {\partial S_2}$, which takes a form of $l(S):=\tr(C \log(P S Q))$. 
Denote $Y=PSQ$ and $L=\log(Y)$, we have
$$
\begin{bmatrix}
    L & d L\\
    0 & L
\end{bmatrix} = \log\left(\begin{bmatrix}
    Y & d Y\\
    0 & Y
\end{bmatrix}\right) = \log\left(\begin{bmatrix}
    P & 0\\
    0 & P
\end{bmatrix}\begin{bmatrix}
    S & d S\\
    0 & S
\end{bmatrix}\begin{bmatrix}
    Q & 0\\
    0 & Q
\end{bmatrix}\right).
$$
Therefore, we get
$$
d l = \left\langle \begin{bmatrix}
    0 & C\\
    0 & 0
\end{bmatrix}, \log\left(\begin{bmatrix}
    P & 0\\
    0 & P
\end{bmatrix}\begin{bmatrix}
    S & d S\\
    0 & S
\end{bmatrix}\begin{bmatrix}
    Q & 0\\
    0 & Q
\end{bmatrix}\right) \right\rangle,
$$
where the inner product is simply the Euclidean inner product. We can plug $d S$ as standard Euclidean basis to obtain an representation of ${d l} / {d S}$, which will take $\mathcal{O}(d^2)$ number of times to cover all the entries. Nevertheless, this provides an implementable way to calculate the Euclidean and Riemannian Hessian.

To summarize, the Euclidean and Riemannian Hessian of $h$ can be calculated as follows.
\begin{equation}\label{eq_karcher_r_hessian}
\begin{aligned}
    & \nabla_S^2 h(S)[V] = -S^{-1} V A^{-1/2} \log(A^{-1/2} S A^{-1/2})A^{1/2} S^{-1} + L, \\
    & H_S (h(S))[V] = S \nabla_S^2 h(S)[V] S +\sym(S \nabla_S h(S) V),
\end{aligned}
\end{equation}
where each entry of matrix $L$ is calculated as follows:
$$
L_{i, j} = \left\langle \begin{bmatrix}
    0 & C\\
    0 & 0
\end{bmatrix}, \log\left(\begin{bmatrix}
    P & 0\\
    0 & P
\end{bmatrix}\begin{bmatrix}
    S & E_{i,j}\\
    0 & S
\end{bmatrix}\begin{bmatrix}
    Q & 0\\
    0 & Q
\end{bmatrix}\right) \right\rangle,
$$
Here the ($i,j$)-th entry of $E_{i,j}\in\RR^{d\times d}$ is one, and all other entries are zeros. Moreover,
\begin{equation*}
\begin{aligned}
    & P = A^{-1/2}, \\
    & Q = A^{-1/2}, \\
    & C = A^{-1/2} V S^{-1} A^{1/2}.
\end{aligned}
\end{equation*}

In the experiment, we test RieBO (Algorithm \ref{algo_bilevel_AID_ITD}) with $d\in\{10, 20\}$ and $n=5$. We repeat each dimension settings for 5 times and plot the average. The algorithm is terminated with $K=200$ rounds of outer iterations, and the inner iteration is also taken to be $T=200$ (the value which we observe a good inner iteration convergence). We take $\alpha_k=10^{-2}$ and $\beta_k=10^{-1}$. Figure \ref{fig:karcher_plots} shows the results of the robust Karcher mean problem \eqref{robust_karcher}. It can be seen from Figure \ref{fig:karcher_plots} that Algorithm \ref{algo_bilevel_robust} can efficiently decrease both the function values and the norm of gradient mappings. We point out here that the computation of the Riemannian Hessian is time consuming by \eqref{eq_karcher_r_hessian} (which is also the reason why we cannot try larger dimensions), yet we remind the reader that this is currently the only formula for calculating it.

\begin{figure}[t!]
    \begin{center}

    \subfigure[Function value]{\includegraphics[width=0.44\textwidth]{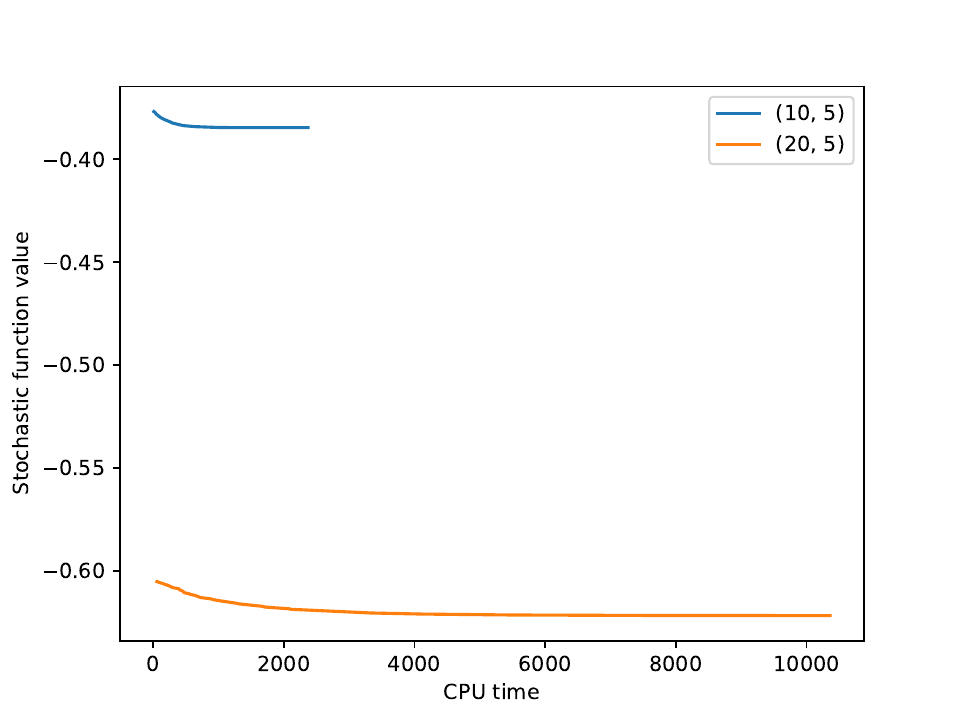}}
    \subfigure[Norm of the approximate gradient mapping]{\includegraphics[width=0.44\textwidth]{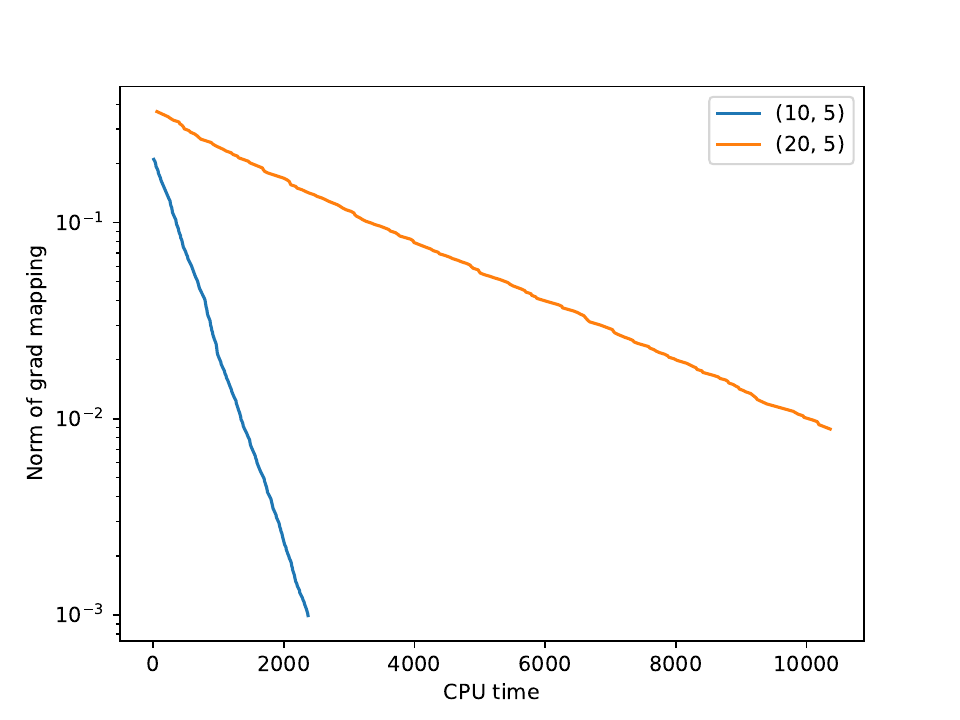}}
    
    \caption{The convergence curve of applying Algorithm \ref{algo_bilevel_robust} to the robust Karcher mean problem \eqref{robust_karcher}. The CPU time is in seconds.}
    \label{fig:karcher_plots}
    \end{center}
\end{figure}

\subsection{Robust maximum likelihood estimation}
For the robust maximum likelihood estimation of the covariance matrix, one seeks to solve:
\begin{equation}\label{robust_cov_mle}
\begin{aligned}
    & \min_{y\in\Delta_n}\ \left\|y-\frac{\mathbf{1}}{n}\right\|^{2} - \sum_{i=1}^{n} y_{i} \mathcal{L}(S; x_i)\\
    & \text{ s.t. } X\in \argmin_{S\in\mathbb{S}^{d}_{++}}\ \sum_{i=1}^{n} y_{i} \mathcal{L}(S; x_i),
\end{aligned}
\end{equation}
where $\mathcal{L}(S; x)$ is the log likelihood of the Gaussian distribution, namely
\begin{equation}\label{eq_Gaussian_mle}
    \mathcal{L}(S;\mathcal{D}):=\frac{1}{2}\logdet(S) + \frac{x^\top S^{-1}x}{2}.
\end{equation}
Note that this lower level problem is geodesically strictly convex (see \cite{sra2015conic}), and thus has a unique solution. The calculations of the Riemannian gradient, Hessian-vector product and cross-derivatives all have closed form solutions (following \cite{petersen2008matrix}). 

In the experiment, we test our algorithm with $d\in\{10, 30,50\}$ and $n=100$. We repeat each dimension settings for 5 times and plot the average. The algorithm is terminated with $K=1000$ rounds of outer iterations, and the inner iteration is still taken to be $T=200$ (again a value which we observe a good inner iteration convergence). We take $\alpha_k=10^{-2}$ and $\beta_k=10^{-1}$. Figure \ref{fig:MLE_plots} shows the results when applying RieBO to the above robust MLE problem with different choices of dimensions. It can be seen from Figure \ref{fig:MLE_plots} that Algorithm \ref{algo_bilevel_robust} can efficiently decrease both the function values and the norm of gradient mappings. Also, here we are able to test and present the results for a much larger dimension due to much faster calculations of Riemannian gradients, Hessian-vector products and cross-derivatives.

\begin{figure}[t!]
    \begin{center}

    \subfigure[Function value]{\includegraphics[width=0.49\textwidth]{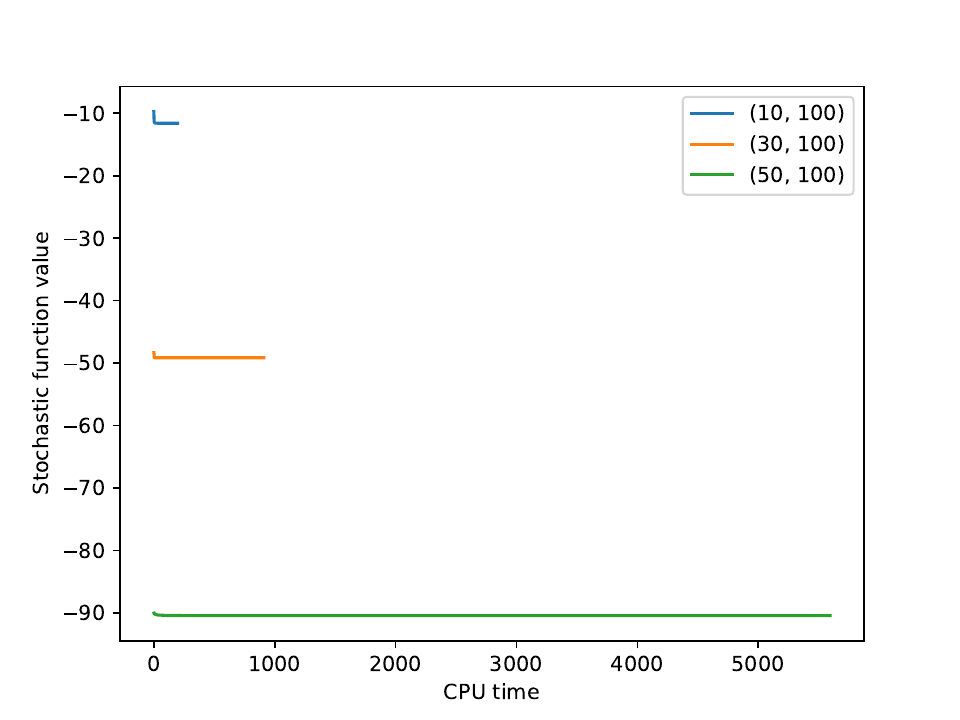}}
    \subfigure[Norm of the approximate gradient mapping]{\includegraphics[width=0.49\textwidth]{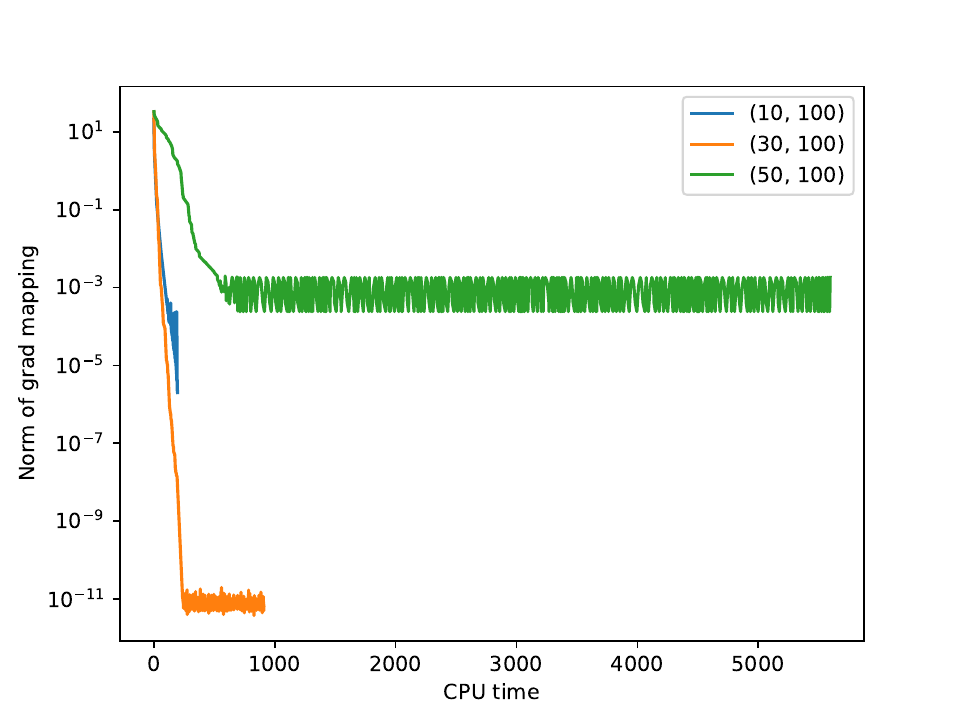}}
    
    \caption{The convergence curve of Algorithm \ref{algo_bilevel_robust} applying to the robust covariance matrix maximum likelihood estimation problem (\ref{robust_cov_mle}) with different choice of $(d,n)$. The CPU time is in seconds.}
    \label{fig:MLE_plots}
    \end{center}
\end{figure}

\section{Conclusion}
We introduced the Riemannian bilevel optimization, a generalization of the traditional Euclidean bilevel optimization. We show that the Riemannian counterparts of Euclidean algorithms in \cite{chen2021tighter,ji2021bilevel} can achieve the same rate of convergence. 

Our work raises several open questions. The first is how we can make the convergence independent of the sectional curvature of the manifold, similar to the results in \cite{cai2023curvature}. It is also worth exploring the last iterate convergence of Riemannian bilevel problem. Last, it still needs investigation to see if there are efficient algorithms that can overcome the difficulty we mentioned in the numerical experiment part to efficiently calculate the Riemannian Hessian-vector product thus enabling large-scale implementation of algorithms for solving the Riemannian bilevel optimization problems.

\bibliographystyle{abbrvnat} 
\bibliography{bibfile}

\end{document}